\documentclass[a4paper]{article}
\usepackage{geometry}
\usepackage[tbtags]{amsmath}
\usepackage{amsthm}
\usepackage{amssymb}
\usepackage{amscd}
\usepackage{cite}
\usepackage{booktabs}
\usepackage[pdftex,colorlinks]{hyperref}
\usepackage{subfigure}
\usepackage{graphicx}
\usepackage{epstopdf}
\usepackage{epsf}
\usepackage{tikz}
\tikzset{>=stealth}
\usepackage{multicol}
\usepackage{multirow}
\usepackage{bm}
\usepackage{caption}

\numberwithin{equation}{section}

\newcommand{\fanshu}[2][0]{\Vert#2\Vert_{#1}}
\newcommand{\banfan}[2][0]{|#2|_{#1}}
\newcommand{\mesh}{\mathcal{T}_h}
\newcommand{\edge}{\mathcal{E}_h}
\newcommand{\face}{\mathcal{F}_h}
\newcommand{\curl}{\nabla\times}

\newcommand{\gcurl}{\nabla\nabla\times}
\newcommand{\divv}{\nabla\cdot}

\newcommand{\jump}[2][]{[\![#2]\!]_{#1}}
\newcommand{\mean}[2][]{\{\!\!\{#2\}\!\!\}_{#1}}

\newtheorem{theorem}{\noindent{\bf Theorem}}[section]

\newtheorem{lemma}{\noindent{Lemma}}[section]

\newtheorem{proposition}{\noindent{\bf Proposition}}[section]
\theoremstyle{remark}
\newtheorem{remark}{\noindent{Remark}}[section]

\begin{document}\large 
\title{{\large  \textbf{Nonconforming finite element approximations and the analysis of Nitsche's method for a singularly perturbed quad-curl problem in three dimensions}}}
\author{\small{Baiju Zhang}\thanks{Beijing Computational Science Research Center, Beijing 100193, China, \tt baijuzhang@csrc.ac.cn} \and \small{Zhimin Zhang}\thanks{Beijing Computational Science Research Center, Beijing 100193, China, {\tt zmzhang@csrc.ac.cn}; Department of Mathematics, Wayne State University, Detroit, MI 48202, USA, \tt ag7761@wayne.edu}}
\date{}\maketitle 

\begin{abstract}
We introduce and analyze a robust nonconforming finite element method for a three dimensional singularly perturbed quad-curl model problem. For the solution of the model problem, we derive proper \textit{a priori} bounds, based on which, we prove that the proposed finite element method is robust with respect to the singular perturbation parameter $\varepsilon$ and the numerical solution is uniformly convergent with order $h^{1/2}$. In addition, we investigate the effect of treating the second boundary condition weakly by Nitsche's method. We show that such a treatment leads to sharper error estimates than imposing the boundary condition strongly when the parameter $\varepsilon< h$. Finally, numerical experiments are provided to illustrate the good performance of the method and confirm our theoretical predictions.
\end{abstract}

\section{Introduction}
Let $\Omega\subset\mathbb{R}^3$ be a contractible Lipschitz domain and $\partial \Omega$ its boundary. We consider the following quad-curl problem
\begin{equation}\label{quadcurlproblem}
\begin{aligned}
\alpha(\nabla\times)^4\pmb{u}+\beta(\nabla\times)^2\pmb{u}+\gamma\pmb{u}&=\pmb{f}\ \text{in }\Omega,\\
\nabla\cdot\pmb{u}&=0\ \text{in }\Omega,\\
\pmb{u}\times\pmb{n}&=0\ \text{on }\partial\Omega,\\
(\nabla\times\pmb{u})\times\pmb{n}&=0\ \text{on }\partial\Omega,
\end{aligned}
\end{equation}
where $\pmb{n}$ is the outward normal unit vector to the boundary $\partial\Omega$, $\pmb{f}\in \pmb{H}(\mathrm{div}^0;\Omega)$ is an external source filed, $\alpha,\ \beta>0$ and $\gamma\ge0$ are parameters.

The above quad-curl problem arises in many areas such as inverse electromagnetic scattering \cite{cakoni2007variational,monk2012finite,sun2011iterative} and magnetohydrodynamics \cite{zheng2011nonconforming}. In recent years, there has been some interest in developing finite element methods (FEMs) for the quad-curl problems. The main difficulty of designing FEMs for the quad-curl problems is the existence of the fourth order operator $(\nabla\times)^4$. To tackle this difficulty, various FEMs have been proposed and analyzed, such as conforming methods \cite{zhang2019Hcurl2,hu2020Simple,zhang2020curlcurl,hu2022afamily}, nonconforming methods \cite{zheng2011nonconforming,huang2022nonconforing,zhang2022anew}, mixed methods \cite{Sun2016MixedFemQuad,Zhang2018MixedQuadCurl,cao2022error}, and discontinuous Galerkin methods \cite{Hong2012DGFourthOrderCurl,Chen2019HDGQuadCurl,Sun2020C0interior}.

The existing literature mainly focuses on the construction of the finite element space and design of the numerical scheme when $\alpha$ and $\beta$ are of the same order, or simply set $\beta=0$. However, according to \cite{zheng2011nonconforming}, $\alpha$ is usually much smaller than either $\beta$ or $\gamma$. This fact brings some challenges in designing robust numerical methods, as studied in the context of fourth order elliptic problems \cite{Wang2006modified,Nilssen2001arobust}. It is well known that the Morley finite element method is not a proper nonconforming method for the following fourth order elliptic singular perturbation problem
\begin{equation}\label{SPbiharmonic}
\begin{aligned}
\varepsilon^2\Delta^2 u-\Delta u&=f,\ \text{in }\Omega,\\
u&=0,\ \text{on }\partial\Omega,\\
\pmb{n}\cdot\nabla u&=0,\ \text{on }\partial\Omega,
\end{aligned}
\end{equation}
where $\varepsilon$ is a real small parameter with $0<\varepsilon\le1$. When $\varepsilon$ tends to zero, the fourth order problem \eqref{SPbiharmonic} formally degenerates to a second order elliptic problem. Since the Morley element is divergent for second order problem\cite{Nilssen2001arobust}, if the Morley element is applied to \eqref{SPbiharmonic}, then the convergence rate of the method will deteriorate when $\varepsilon$ closes to zero. To overcome this difficulty, one can modify the formulation of the Morley method \cite{Wang2006modified} or alert the element itself \cite{Nilssen2001arobust}. It is natural to ask whether the existing nonconforming finite element methods are robust when $\alpha\ll\beta$ for the quad-curl problem \eqref{quadcurlproblem}, and if no, how to construct robust ones. This paper attempts to answer this question. To the best of the authors' knowledge, this is the first paper for the case $\alpha\ll\beta$ for the quad-curl problem \eqref{quadcurlproblem}.

For the convenience of description and analogy with \eqref{SPbiharmonic}, we might as well restrict our attention to the following singularly perturbed quad-curl problem
\begin{equation}\label{SPquadcurlproblem}
\begin{aligned}
\varepsilon^2(\nabla\times)^4\pmb{u}+(\nabla\times)^2\pmb{u}&=\pmb{f}\ \text{in }\Omega,\\
\nabla\cdot\pmb{u}&=0\ \text{in }\Omega,\\
\pmb{u}\times\pmb{n}&=0\ \text{on }\partial\Omega,\\
(\nabla\times\pmb{u})\times\pmb{n}&=0\ \text{on }\partial\Omega,
\end{aligned}
\end{equation}
where $\varepsilon$ is a real small parameter with $0<\varepsilon\le1$. In numerical tests, we observe that when $\varepsilon$ goes to zero, the convergence rate deteriorates if the total nonconforming finite element methods \cite{zheng2011nonconforming,huang2022nonconforing} are applied to the singularly perturbed quad-curl problem \eqref{SPquadcurlproblem}. The main reason for this deterioration is the fact that the nonconforming methods diverge for the Maxwell's equation as mentioned in \cite[Remark 4.10]{huang2022nonconforing}, while the quad-curl problem \eqref{SPquadcurlproblem} formally degenerates to the Maxwell's equation when $\varepsilon$ tends to zero. In the fourth order elliptic context \cite{Nilssen2001arobust}, Nilssen et al. proposed a robust nonconforming $H^2$-element with $H^1$-conformity for \eqref{SPbiharmonic}. This makes us conjecture that if we can construct a nonconforming finite element with $\pmb{H}(\mathrm{curl})$-continuity, we should be able to obtain a robust method for the singularly perturbed quad-curl problem \eqref{SPquadcurlproblem}. Fortunately, we have constructed such a family of nonconforming elements in our previous work \cite{zhang2022anew}. To get robust error estimates, it is still necessary to derive some \textit{a priori} bounds with respect to $\varepsilon$ for the solution of \eqref{SPquadcurlproblem} and investigate the relationship between the solution of \eqref{SPquadcurlproblem} and the following reduced problem
\begin{equation}\label{reducedproblem}
\begin{aligned}
(\nabla\times)^2\overline{\pmb{u}}&=\pmb{f}\ \text{in }\Omega,\\
\nabla\cdot\overline{\pmb{u}}&=0\ \text{in }\Omega,\\
\overline{\pmb{u}}\times\pmb{n}&=0\ \text{on }\partial\Omega.
\end{aligned}
\end{equation}
The corresponding results are given in this paper. As far as we are aware, these results are not available in the existing literature. Based on these results, we prove that the proposed nonconforming method is convergent uniformly in $\varepsilon$ in the energy norm with bounds of order $h^{1/2}$.

Another purpose of this paper is to analyze the effect of enforcing the second boundary condition in \eqref{SPquadcurlproblem} weakly by the Nitsche's method. In \cite{guzman2012spfourth}, Guzm\'{a}n et al. applied the Nitsche's method to the singularly perturbed fourth order problem \eqref{SPbiharmonic}. More precisely, the Neumann boundary condition in \eqref{SPbiharmonic} is imposed weakly in their numerical scheme. They show that such treatment is superior when the parameter $\varepsilon$ is smaller than the mesh size $h$ and obtain sharper error estimates. Inspired by this, we ask whether this method can be extended to the singularly perturbed quad-curl problem \eqref{SPquadcurlproblem}. It turns out that the answer is yes, and by some regularity assumptions on the solution of the reduced problem \eqref{reducedproblem} we obtain sharper error estimates than that imposing boundary conditions strongly when the parameter $\varepsilon< h$.

The rest of this paper is organized as follows: In Section \ref{sectionpreliminary}, we introduce some notations. We also derive some \textit{a priori} estimates with respect to $\varepsilon$ for the solution of \eqref{SPquadcurlproblem} and investigate the relationship between the solution of \eqref{SPquadcurlproblem} and the reduced problem \eqref{reducedproblem}. In Section \ref{sectionFEM}, we introduce a family of nonconforming elements for the quad-curl problem in 3D. We also define the related interpolation and study its approximation properties. In Section \ref{sectionApplication}, we propose and analyze the finite element method with weakly imposed boundary conditions. We prove that the proposed method is robust with respect to the parameter $\varepsilon$, and corresponding error estimates are sharper than that imposing boundary conditions strongly. In Section \ref{sectionnumerical}, we provide some numerical results to verify our theoretical analysis.

Throughout the paper, we use the notations $A\lesssim B$ and $A\gtrsim B$ to represent the inequalities $A\le CB$ and $A\ge CB$, respectively, where $C$ is independent of $h,\varepsilon$ or any penalty parameters.

\section{Preliminaries}\label{sectionpreliminary}
\subsection{Notation}
Let $\Omega\subset\mathbb{R}^3$ be a contractible Lipschitz domain and $\partial \Omega$ its boundary. For a positive integer, we utilize common notation for the Sobolev spaces $H^m(D)$ or $H^m_0(D)$ on a simply connected subdomain $D\subset\Omega$, equipped with the norm $\fanshu[m,D]{\cdot}$ and seminorm $\banfan[m,D]{\cdot}$. If $D=\Omega$, the subscript will be omitted. Conventionally, we write $L^2(D)$ instead of $H^0(D)$.

Let $\pmb{u}=(u_1,u_2,u_3)^T$ and $\pmb{w}=(w_1,w_2,w_3)^T$, where the superscript $T$ denotes the transpose, then we denote
\begin{align*}
\divv\pmb{u}&=\frac{\partial u_1}{\partial x_1}+\frac{\partial u_2}{\partial x_2}+\frac{\partial u_3}{\partial x_3},\\
\pmb{u}\times\pmb{w}&=(u_2w_3-u_3w_2,u_1w_3-u_3w_1,u_1w_2-u_2w_1)^T,\\
\curl\pmb{u}&=(\frac{\partial u_3}{\partial x_2}-\frac{\partial u_2}{\partial x_3},\frac{\partial u_1}{\partial x_3}-\frac{\partial u_3}{\partial x_1},\frac{\partial u_2}{\partial x_1}-\frac{\partial u_1}{\partial x_2})^T,\\
(\curl)^2\pmb{u}&= \curl\curl\pmb{u}.
\end{align*}

On a simply-connected sub-domain $D\subset\Omega$ we write
\begin{align}
L_0^2({D})&=\{q\in L^2({D}): \int_{{D}}qd\pmb{x}=0\},\\
\pmb{H}(\mathrm{div};{D})&=\{\pmb{v}\in \pmb{L}^2({D}): \divv \pmb{v}\in L^2({D})\},\\
\pmb{H}(\mathrm{div}^0;{D})&=\{\pmb{v}\in \pmb{H}(\mathrm{div};{D}):\divv\pmb{v}=0 \text{ in }{D}\},\\
\pmb{H}_0(\mathrm{div};{D})&=\{\pmb{v}\in \pmb{H}(\mathrm{div};{D})): \pmb{v}\cdot\pmb{n}=0\ \text{ on }\partial{D}\},\\
\pmb{H}(\mathrm{curl};{D})&=\{\pmb{v}\in \pmb{L}^2({D}): \curl \pmb{v}\in \pmb{L}^2({D})\},\\
\pmb{H}_0(\mathrm{curl};{D})&=\{\pmb{v}\in \pmb{H}(\mathrm{curl};{D}):\pmb{v}\times\pmb{n}=0 \text{ on }\partial{D}\},\\
\pmb{H}^1(\mathrm{curl};{D})&=\{\pmb{v}\in \pmb{H}^1({D}): \curl\pmb{v}\in \pmb{H}^1({D})\},\\
\pmb{H}^1_0(\mathrm{curl};{D})&=\{\pmb{v}\in \pmb{H}^1(\mathrm{curl};{D})\cap\pmb{H}_0^1({D}):\ \curl\pmb{v}=0 \text{ on }\partial{D}\},\\
\pmb{H}(\mathrm{grad}\ \mathrm{curl};{D})&=\{\pmb{v}\in \pmb{L}^2({D}): \curl\pmb{v}\in\pmb{H}^1({D})\},\\
\pmb{H}_0(\mathrm{grad}\ \mathrm{curl};{D})&=\{\pmb{v}\in\pmb{H}(\mathrm{grad}\ \mathrm{curl};{D}):\pmb{v}\times\pmb{n}=0, \ \curl\pmb{v}=0 \text{ on }\partial{D}\},
\end{align}
where $\pmb{n}$ is the unit outward normal vector to the boundary $\partial D$.

Let $\{\mesh\}_{h>0}$ be a family of quasi-uniform simplicial triangulation of $\Omega$ with $h_K=\mathrm{diam}(K)$ for all $K\in\mesh$ and $h=\max_{K\in\mesh}h_K$. We denote by $\face$ the faces, $\face^i$ the interior faces and $\face^b$ the boundary faces in $\mesh$. Furthermore, we define the patch of $F\in\face$ and $K\in\mesh$ as
\begin{align*}
\mathcal{T}_F=\{K\in\mesh:\ F\subset\partial K\},\qquad \mathcal{T}_K=\{K'\in\mesh:\ \partial K'\cap\partial K\neq\emptyset\}.
\end{align*}

Given $K\in\mesh$ and a face $F\subset\partial K$, we denote by $\lambda_F$ the barycentric coordinate of $K$ which vanishes on $F$. The element bubble and face bubbles are given by
\begin{align}
b_K=\prod_F\lambda_F,\quad b_F=\prod_{G\neq F}\lambda_G,
\end{align}
where the product runs over the faces of $K$.

For a given simplex $S$ in $\mathbb{R}^3$ and $m\ge0$, the vector-valued polynomials are defined as $\pmb{P}_m(S)=[P_m(S)]^3$, where $P_m(S)$ is the space of polynomials defined of $S$ of degree less than or equal to $m$. We also set $\pmb{P}_m(S)$ and $P_m(S)$ to be the empty set if $m<0$.

For interior and boundary inner products we will use the following notation:
\begin{align*}
(\pmb{u},\pmb{v})_K=\int_K\pmb{u}\cdot\pmb{v}dx,\quad \langle p,q\rangle_F=\int_F pqds.
\end{align*}
We denote by $\pmb{n}_F$ the unit outward normal vector to a face $F$ of $K$.

We define the tangential and normal jump operators. If $F\in \face^i$ with $F=K^+\cap K^-$, then we set
\begin{align*}
\jump{\pmb{v}\times\pmb{n}}|_F&=(\pmb{v}|_{K^+}\times\pmb{n}_{K^+})|_F+(\pmb{v}|_{K^-}\times\pmb{n}_{K^-})|_F,\\
\jump{\pmb{v}\cdot\pmb{t}}|_F&=(\pmb{v}|_{K^+}\cdot\pmb{t}_{K^+})|_F+(\pmb{v}|_{K^-}\cdot\pmb{t}_{K^-})|_F,\\
\jump{\pmb{v}\cdot\pmb{n}}|_F&=(\pmb{v}|_{K^+}\cdot\pmb{n}_{K^+})|_F+(\pmb{v}|_{K^-}\cdot\pmb{n}_{K^-})|_F,
\end{align*}
where $\pmb{n}_{K^{\pm}}$ is the unit outward normal to $\partial K^{\pm}$ and $\pmb{t}_{K^{\pm}}$ is the unit outward tangent to $\partial K^{\pm}$.

If $F\in \face^b$ is a boundary face with $F\subset\partial K$, then we set
\begin{align*}
\jump{\pmb{v}\times\pmb{n}}|_F&=(\pmb{v}|_{K}\times\pmb{n}_{K})|_F,\\
\jump{\pmb{v}\cdot\pmb{t}}|_F&=(\pmb{v}|_{K}\cdot\pmb{t}_{K})|_F,\\
\jump{\pmb{v}\cdot\pmb{n}}|_F&=(\pmb{v}|_{K}\cdot\pmb{n}_{K})|_F.
\end{align*}
We also define the average operator as
\begin{align}
\mean{\pmb{v}\times\pmb{n}}|_F&=\frac{1}{2}\left(\pmb{v}|_{K^+}\times\pmb{n}_F|_F+\pmb{v}|_{K^-}\times\pmb{n}_F|_F\right)\quad\text{ if }F\in\face^i\text{ with } F=K^+\cap K^- \\
\mean{\pmb{v}\times\pmb{n}}|_F&=\pmb{v}\times\pmb{n}|_F\quad\text{ if }F\in\face^b
\end{align}

In the next sections, we will often use the following version of the trace inequality (see
\cite[Theorem 1.6.6]{brenner2008fembook} and \cite[Proposition 1]{guzman2012spfourth})
\begin{proposition}
For any simply connected domain $D$ with piece-wise smooth boundary $\partial D$, we have
\begin{align}
\fanshu[0,\partial D]{v}\lesssim\fanshu[0,D]{v}^{1/2}\fanshu[1,D]{v}^{1/2},\ \forall v\in H^1(D).\label{trace1}
\end{align}
In particular, by a standard scaling we get the following estimates on any simplex $K$
\begin{align}
\fanshu[0,\partial K]{v}&\lesssim h_K^{-1/2}\fanshu[0,K]{v}+\fanshu[0,D]{v}^{1/2}\fanshu[0,K]{\nabla v}^{1/2},\ \forall v\in H^1(K),\label{trace2}\\
\fanshu[0,\partial K]{v}&\lesssim h_K^{-1/2}\fanshu[0,K]{v}+h_K^{1/2}\fanshu[0,K]{\nabla v},\ \forall v\in H^1(K).\label{trace3}
\end{align}
In addition we will also need a standard inverse estimates. Let $q$ be fixed. Then for all $v\in P^q(K)$,
\begin{align}
\fanshu[0,\partial K]{v}\lesssim h_K^{-1/2}\fanshu[0,K]{v}.\label{trace4}
\end{align}
\end{proposition}

\subsection{Regularity estimates of the singularly perturbed quad-curl problem}
This subsection is devoted to investigate the following regularity estimates of the singularly perturbed quad-curl problem \eqref{SPquadcurlproblem}.
\begin{lemma}
Assume domain $\Omega$ is convex. Let $\pmb{u}$ be the solution to \eqref{SPquadcurlproblem}, and $\overline{\pmb{u}}$ the solution to the reduced problem \eqref{reducedproblem}. Then the following estimates hold
\begin{align}
\fanshu[2]{\pmb{u}}+\varepsilon\fanshu[2]{\curl\pmb{u}}&\lesssim\varepsilon^{-1/2}\fanshu[0]{\pmb{f}},\label{regularitySP1}\\
\fanshu[0]{\curl(\pmb{u}-\overline{\pmb{u}})}&\lesssim\varepsilon^{1/2}\fanshu[0]{\pmb{f}},\label{regularitySP2}\\
\fanshu[0]{\pmb{u}-\overline{\pmb{u}}}&\lesssim\varepsilon\fanshu[0]{\pmb{f}},\label{regularitySP3}\\
\fanshu[1]{\pmb{u}-\overline{\pmb{u}}}&\lesssim\varepsilon^{1/2}\fanshu[0]{\pmb{f}}.\label{regularitySP4}
\end{align}
\end{lemma}
\begin{proof}
We will establish the bounds above by similar arguments employed in \cite[Lemma 5.1]{Nilssen2001arobust}. Let us recall the following regularity estimate of \eqref{reducedproblem}
\begin{align}
\fanshu[2]{\overline{\pmb{u}}}\lesssim \fanshu[0]{\pmb{f}}.\label{regularityreduced}
\end{align}
The above estimate is a direct consequence of \cite[Theorem 3.5]{Zhang2018MixedQuadCurl}. Since
\begin{align}
(\curl)^4\pmb{u}=\varepsilon^{-2}(\curl)^2(\overline{\pmb{u}}-\pmb{u}),\label{spquadcurlminusreduced}
\end{align}
 then it follows from
\cite[Lemma A.1]{huang2022nonconforing} that
\begin{align}
\fanshu[2]{(\curl)^2\pmb{u}}\lesssim\varepsilon^{-2}\fanshu[-1]{(\curl)^2(\overline{\pmb{u}}-\pmb{u})}\lesssim\varepsilon^{-2}\fanshu[0]{\curl(\overline{\pmb{u}}-\pmb{u})}.\label{estimateuh}
\end{align}
Multiplying both sides of \eqref{spquadcurlminusreduced} by $\pmb{u}-\overline{\pmb{u}}$ and integrating by parts, we get
\begin{align}
\varepsilon^2\fanshu[0]{(\curl)^2\pmb{u}}+\fanshu[0]{\curl(\pmb{u}-\overline{\pmb{u}})}^2=\langle\varepsilon^2(\curl)^2\pmb{u},\pmb{n}\times(\curl\overline{\pmb{u}})\rangle+(\varepsilon^2(\curl)^2\pmb{u},\pmb{f}).\label{equality1}
\end{align}
Using the trace inequality \eqref{trace1} and \eqref{regularityreduced}, we get
\begin{align*}
\fanshu[0,\partial\Omega]{(\curl)^2\pmb{u}}\lesssim\fanshu[0]{(\curl)^2\pmb{u}}^{1/2}\fanshu[1]{(\curl)^2\pmb{u}}^{1/2}
\end{align*}
and
\begin{align*}
\fanshu[0,\partial\Omega]{\curl\overline{\pmb{u}}}\lesssim\fanshu[0]{\pmb{f}}.
\end{align*}
Furthermore, it is clear to see that
\begin{align*}
(\varepsilon^2(\curl)^2\pmb{u},\pmb{f})\le \frac{\varepsilon^2}{2}(\fanshu[0]{\curl)^2\pmb{u}}^2+\fanshu[0]{\pmb{f}}^2).
\end{align*}
Then by the Young's inequality we obtain
\begin{align}
\langle\varepsilon^2(\curl)^2\pmb{u},\pmb{n}\times(\curl\overline{\pmb{u}})\rangle\lesssim \varepsilon^3\delta\fanshu[0]{(\curl)^2\pmb{u}}\fanshu[1]{(\curl)^2\pmb{u}}+\frac{\varepsilon}{\delta}\fanshu[0]{\pmb{f}}^2,
\end{align}
where $\delta>0$.
From \eqref{estimateuh} we derive
\begin{align}
\varepsilon^3\fanshu[0]{(\curl)^2\pmb{u}}\fanshu[1]{(\curl)^2\pmb{u}}&\le\frac{1}{2}(\varepsilon^2\fanshu[0]{(\curl)^2\pmb{u}}^2+\varepsilon^4\fanshu[1]{(\curl)^2\pmb{u}}^2)\notag\\
&\lesssim\frac{1}{2}(\varepsilon^2\fanshu[0]{(\curl)^2\pmb{u}}^2+\fanshu[0]{\curl(\pmb{u}-\overline{\pmb{u}})}).\label{inequality1}
\end{align}
Combing \eqref{equality1}-\eqref{inequality1}, we have
\begin{align}
\varepsilon^2\fanshu[0]{(\curl)^2\pmb{u}}+\fanshu[0]{\curl(\pmb{u}-\overline{\pmb{u}})}^2\lesssim\varepsilon\fanshu[0]{\pmb{f}}^2.\label{inequality2}
\end{align}
It follows from \eqref{estimateuh} and \eqref{inequality2} that
\begin{align}
\fanshu[2]{\curl\pmb{u}}\lesssim\varepsilon^{-3/2}\fanshu[0]{\pmb{f}}.\label{estimatecurlu}
\end{align}
According to \cite[Lemma 3.6]{zhang2022anew}, we have
\begin{align}
\fanshu[2]{\pmb{u}}\lesssim\fanshu[1]{\curl\pmb{u}}\lesssim\fanshu[0]{(\curl)^2\pmb{u}},\label{inequality3}
\end{align}
where we have used the Friedrichs inequality and the fact that
\begin{align}
\fanshu[0]{(\curl)^2\pmb{u}}=\fanshu[0]{\nabla\curl\pmb{u}}.
\end{align}
By \eqref{inequality2} and \eqref{inequality3}, we derive
\begin{align}
\fanshu[2]{\pmb{u}}\lesssim\varepsilon^{-1/2}\fanshu[0]{\pmb{f}}.\label{inequality4}
\end{align}
The estimates \eqref{regularitySP1} and \eqref{regularitySP2} then follow from \eqref{inequality2}, \eqref{estimatecurlu} and \eqref{inequality4}.

We now turn to the $L^2$ estimate \eqref{regularitySP3}. To this end, we use a duality argument. For arbitrary $\pmb{v}\in\pmb{L}^2(\Omega)$, $\pmb{\psi}$ be the solution to the following problem:
\begin{equation}\label{dualmaxwell}
\begin{aligned}
(\nabla\times)^2\pmb{\psi}&=\pmb{v}\ \text{in }\Omega,\\
\nabla\cdot\pmb{\psi}&=0\ \text{in }\Omega,\\
\pmb{\psi}\times\pmb{n}&=0\ \text{on }\partial\Omega.
\end{aligned}
\end{equation}
Similar to \eqref{regularityreduced} we have
\begin{align}
\fanshu[2]{\pmb{\psi}}\lesssim \fanshu[0]{\pmb{v}}.\label{regularitymaxwelldual}
\end{align}
By \eqref{dualmaxwell}, \eqref{SPquadcurlproblem}, \eqref{reducedproblem} and integration by parts, we obtain
\begin{align}
(\pmb{u}-\overline{\pmb{u}},\pmb{v})=(\curl(\pmb{u}-\overline{\pmb{u}}),\curl\pmb{\psi})=(\varepsilon^2(\curl)^2\pmb{u},(\curl)^2\pmb{\psi})+\langle\varepsilon^2\pmb{n}\times(\curl)^2\pmb{u},\curl\pmb{\psi}\rangle.
\end{align}
Then by the trace inequality, the estimate \eqref{regularitySP1} and \eqref{regularitymaxwelldual}, we have
\begin{align*}
(\pmb{u}-\overline{\pmb{u}},\pmb{v})&\lesssim \varepsilon^2(\fanshu[0]{(\curl)^2\pmb{u}}+\fanshu[0]{(\curl)^2\pmb{u}}^{1/2}\fanshu[1]{(\curl)^2\pmb{u}}^{1/2})\fanshu[2]{\pmb{\psi}}\\
&\lesssim\varepsilon^2(\varepsilon^{-1/2}+\varepsilon^{-1/4}\varepsilon^{-3/4})\fanshu[0]{\pmb{f}}\fanshu[0]{\pmb{v}}\\
&\lesssim\varepsilon\fanshu[0]{\pmb{f}}\fanshu[0]{\pmb{v}},
\end{align*}
which leads to \eqref{regularitySP3}.

According to \cite[Theorem3.7]{Vivette1986finite}, recalling that $\nabla\cdot\pmb{u}=\nabla\cdot\overline{\pmb{u}}=0$, we derive
\begin{align}
\fanshu[1]{\pmb{u}-\overline{\pmb{u}}}\lesssim\fanshu[0]{\pmb{u}-\overline{\pmb{u}}}+\fanshu[0]{\curl(\pmb{u}-\overline{\pmb{u}})}.
\end{align}
Therefore, \eqref{regularitySP4} follows from \eqref{regularitySP2} and \eqref{regularitySP3}.
\end{proof}

\section{A family of nonconforming curl-curl finite element in 3D}\label{sectionFEM}
\subsection{The local space}
In this section we introduce a family of nonconforming elements for the singularly perturbed quad-curl problem \eqref{SPquadcurlproblem}. Since our the nonconforming elements are be based on $\pmb{H}(\mathrm{curl};\Omega)$ finite element spaces plus some bubble functions, we first review some well-known elements. For $K\in\mesh$, let $\pmb{N}^k(K)$ denote the local N\'{e}d\'{e}lec space of the first kind,
\begin{align}
\pmb{N}^k(K)=\pmb{P}^{k}(K)+\{\pmb{v}\in\pmb{P}^{k+1}(K):\pmb{v}\cdot\pmb{x}=0\}\quad(k\ge1),\label{Nedelec1}
\end{align}
and let $\pmb{M}^k(K)$ denote the local BDM space
\begin{align}
\pmb{M}^{k}(K)=\pmb{P}^{k}(K)\quad(k\ge1).\label{BDM}
\end{align}

We then define the local space for the 3D quad-curl problem as
\begin{align}
\pmb{U}^k(K)=\pmb{N}^k(K)+b_K\pmb{Q}^{k-1}(K),
\end{align}
where
\begin{align}
\pmb{Q}^{k-1}(K)=\sum_{F}b_F\pmb{Q}_F^{k-1}(K),
\end{align}
and
\begin{align}
\pmb{Q}_F^{k-1}(K)=\{&\pmb{q}\times\pmb{n}_F\in\pmb{P}^{k-1}(K)\times\pmb{n}_F:\notag\\
&\quad(\pmb{q}\times\pmb{n}_F,b_Kb_F(\pmb{w}\times\pmb{n}_F))_K=0\quad\forall\pmb{w}\in\pmb{P}^{k-2}(K) \}.\label{definitionQF}
\end{align}
The space $b_K\pmb{Q}^{k-1}(K)$ was introduced by \cite{guzman2012afamily}.

We equip the local space $\pmb{U}^k(K)$  with the following degrees of freedom
\begin{align}
&\langle\pmb{u}\cdot\pmb{t}_e,\kappa\rangle_e\qquad\qquad\qquad\text{for all }\kappa\in P^{k}(e) \text{ and edge }e\text{ of }K,\label{nonconDOFs1}\\
&\langle\pmb{u}\times\pmb{n}_F,\pmb{\mu}\rangle_F\qquad\qquad\ \text{ for all }\pmb{\mu}\in \pmb{P}^{k-1}(F) \text{ and faces }F\text{ of }K,\label{nonconDOFs2}\\
&\langle(\curl\pmb{u})\times\pmb{n}_F,\pmb{\chi}\rangle_F\quad\ \text{ for all }\pmb{\chi}\in \pmb{P}^{k-1}(F) \text{ and faces }F\text{ of }K,\label{nonconDOFs3}\\
&(\pmb{u},\pmb{\rho})_K\qquad\qquad\qquad\quad\text{for all }\pmb{\rho}\in \pmb{P}^{k-2}(K).\label{nonconDOFs4}
\end{align}
We proved in \cite[Theorem 3.1]{zhang2022anew} that any function $\pmb{u}\in\pmb{U}^k(K)$ can be uniquely determined by the degrees of freedom \eqref{nonconDOFs1}-\eqref{nonconDOFs4}.

\subsection{The global space and interpolation}
We have defined the local finite element spaces, then we can define the associated global spaces as
\begin{align}
\pmb{U}_h = \{\pmb{u}\in& \pmb{H}_0(\mathrm{curl};\Omega):\ \pmb{u}|_K\in\pmb{U}^k(K),\ \forall K\in\mesh,\notag\\
&\langle\jump{(\curl\pmb{u})\times\pmb{n}},\pmb{\chi}\rangle_F=0\ \forall\pmb{\chi}\in\pmb{P}^{k-1}(F),\ \forall F\in\face^i\},\label{definitionUh}\\
\pmb{U}_{h,0} = \{\pmb{u}\in& \pmb{U}_h\cap\pmb{H}_0(\mathrm{curl};\Omega): \langle\jump{(\curl\pmb{u})\times\pmb{n}},\pmb{\chi}\rangle_F=0\ \forall\pmb{\chi}\in\pmb{P}^{k-1}(F),\ \forall F\in\face^b\},\label{definitionUh0}\\
\pmb{N}_h = \{\pmb{u}\in& \pmb{H}_0(\mathrm{curl};\Omega):\ \pmb{u}|_K\in\pmb{N}^k(K),\ \forall K\in\mesh\},\\
\pmb{M}_h = \{\pmb{v}\in& \pmb{H}_0(\mathrm{div};\Omega):\ \pmb{v}|_K\in\pmb{M}^k(K),\ \forall K\in\mesh\},\\
W_{h,0} = \{w\in& H_0^1(\Omega):\ w|_K\in P^{k+1}(K),\ \forall K\in\mesh\}.
\end{align}

We now turn to investigate the interpolation operator associated to $\pmb{U}_h$. The degrees of freedom \eqref{nonconDOFs1}-\eqref{nonconDOFs4} naturally induce an interpolation operator $\pmb{\Pi}_{U}:\pmb{H}^1(\mathrm{curl};\Omega)\rightarrow\pmb{U}_h$ \cite{zhang2022anew}. However, such an operator may not satisfy the following property
\[\fanshu[0]{\curl\pmb{\Pi}_{U}\pmb{u}}\lesssim\fanshu[0]{\curl\pmb{u}},\]
and such a property is quite important for proving the uniform convergence with respect to $\varepsilon$. Therefore we need to define another operator.

To this end, we recall the projection operators introduced in \cite{ern2016molification}. According to \cite[Theorem 6.5]{ern2016molification}, there exist operators $\mathcal{J}_h^N: \pmb{H}_0(\mathrm{curl};{\Omega})\rightarrow\pmb{N}_h$ and $\mathcal{J}_h^M: \pmb{H}_0(\mathrm{div};{\Omega})\rightarrow\pmb{M}_h$ such that the following properties hold:
\begin{itemize}
  \item[(i)] $P(\mesh)$ is point-wise invariant under $\mathcal{J}_h$, where $P(\mesh)=\pmb{N}_h$ if $\mathcal{J}_h=\mathcal{J}_h^N$, and $P(\mesh)=\pmb{M}_h$ if $\mathcal{J}_h=\mathcal{J}_h^M$.
  \item[(ii)] There is $c$, uniform with respect to $h$, such that
  \begin{equation}
  \fanshu[\mathcal{L}(\pmb{L}^2(\Omega);\pmb{L}^2(\Omega))]{\mathcal{J}_h}\le c,\label{JhBounded}
  \end{equation}
  and
  \begin{equation}
  \fanshu[0]{\pmb{v}-\mathcal{J}_h\pmb{v}}\le C\inf_{\pmb{v}_h\in P(\mesh)}\fanshu[0]{\pmb{v}-\pmb{v}_h},\ \forall \pmb{v}\in\pmb{L}^2(\Omega).\label{Jhinf}
  \end{equation}
  \item[(iii)] The following diagram is commutative:
  \begin{equation}\label{Jhdiagram}
  \begin{CD}
  \pmb{H}_0(\mathrm{curl};{\Omega})@>\curl>>\pmb{H}_0(\mathrm{div};{\Omega})\\
  @VV\mathcal{J}_h^NV@VV\mathcal{J}_h^MV\\
  \pmb{N}_h@>\curl>>\pmb{M}_h
  \end{CD}
  \end{equation}
\end{itemize}

Furthermore, we can prove the following lemma.
\begin{lemma}\label{lemmaJhM}
Let integer $s$ satisfy $0\le s\le k+1$. Then for any $\pmb{w}\in\pmb{H}^s(\Omega)\cap\pmb{H}_0(\mathrm{div};\Omega)$, there holds
\begin{align}
\left(\sum_{K\in\mesh}\fanshu[m,K]{\pmb{w}-\mathcal{J}_h^M\pmb{w}}^2\right)^{1/2}\lesssim h^{s-m}\fanshu[s]{\pmb{w}},\ 0\le m\le s.
\end{align}
\end{lemma}
\begin{proof}
If $s=m=0$, then using \eqref{JhBounded} we immediately derive
\begin{align*}
\fanshu[0]{\pmb{w}-\mathcal{J}_h^M\pmb{w}}\le \fanshu[0]{\pmb{w}}+\fanshu[0]{\mathcal{J}_h^M\pmb{w}}\lesssim\fanshu[0]{\pmb{w}}.
\end{align*}

We now consider $s\ge 1$. Let $\pmb{\Pi}_h^M:\pmb{H}_0(\mathrm{div};\Omega)\cap\pmb{H}^1(\Omega)\rightarrow\pmb{M}_h $ be the standard BDM interpolation \cite{boffi2013mixed}. Using the triangle inequality, inverse estimate, and
\cite[Theorem 11.13]{ern2021finteI}, we have
\begin{align*}
\fanshu[m,K]{\pmb{w}&-\mathcal{J}_h^M\pmb{w}}\le\fanshu[m,K]{\pmb{w}-\pmb{\Pi}_h^M\pmb{w}}+\fanshu[m,K]{\pmb{\Pi}_h^M\pmb{w}-\mathcal{J}_h^M\pmb{w}}\\
&\lesssim h_K^{s-m}\fanshu[s,K]{\pmb{w}}+h_K^{-m}\fanshu[0,K]{\pmb{\Pi}_h^M\pmb{w}-\mathcal{J}_h^M\pmb{w}}\\
&\le h_K^{s-m}\fanshu[s,K]{\pmb{w}}+h_K^{-m}(\fanshu[0,K]{\pmb{\Pi}_h^M\pmb{w}-\pmb{w}}+\fanshu[0,K]{\pmb{w}-\mathcal{J}_h^M\pmb{w}}).
\end{align*}
By the above estimate, \eqref{Jhinf} and the approximation property of $\pmb{\Pi}_h^M$, we obtain
\begin{align*}
\left(\sum_{K\in\mesh}\fanshu[m,K]{\pmb{w}-\mathcal{J}_h^M\pmb{w}}^2\right)^{1/2}&\lesssim h^{s-m}\fanshu[s]{\pmb{w}}+h^{-m}(\fanshu[0]{\pmb{\Pi}_h^M\pmb{w}-\pmb{w}}+\fanshu[0]{\pmb{w}-\mathcal{J}_h^M\pmb{w}})\\
&\lesssim h^{s-m}\fanshu[s]{\pmb{w}}+h^{-m}\fanshu[0]{\pmb{\Pi}_h^M\pmb{w}-\pmb{w}}\\
&\lesssim h^{s-m}\fanshu[s]{\pmb{w}},
\end{align*}
which ends our proof.
\end{proof}

We are now in a position to state the construction of operator related to $\pmb{U}_h$.
\begin{lemma}\label{lemmaIh}
There exists an operator $\pmb{I}_h:\pmb{H}_0(\mathrm{curl};\Omega)\rightarrow\pmb{U}_h$ such that for any $\pmb{v}\in\pmb{H}^s(\Omega)\cap\pmb{H}_0(\mathrm{curl};\Omega)$ and $\curl\pmb{v}\in\pmb{H}^s(\Omega)$ with $0\le s\le k+1$, there hold
\begin{align}
\fanshu[0]{\pmb{v}-\pmb{I}_h\pmb{v}}&\lesssim h^{s}(\fanshu[s]{\pmb{v}}+\fanshu[s]{\curl\pmb{v}}),\label{estimateIh1}\\
(\sum_{K\in\mesh}\fanshu[m,K]{\curl\pmb{v}-\curl \pmb{I}_h\pmb{v}}^2)^{1/2}&\lesssim h^{s-m}\fanshu[s]{\curl\pmb{v}}\quad 0\le m\le s,\label{estimateIh2}
\end{align}
Moreover, there exists $\pmb{I}_h^0:\pmb{H}_0(\mathrm{curl};\Omega)\rightarrow\pmb{U}_{h,0}$ such that if $\pmb{v}\in\pmb{H}_0(\mathrm{grad}\ \mathrm{curl};\Omega)$
\begin{align}
\fanshu[0]{\pmb{v}-\pmb{I}_h^0\pmb{v}}&\lesssim h^{s}(\fanshu[s]{\pmb{v}}+\fanshu[s]{\curl\pmb{v}}),\label{estimateIh01}\\
(\sum_{K\in\mesh}\fanshu[m,K]{\curl\pmb{v}-\curl \pmb{I}_h^0\pmb{v}}^2)^{1/2}&\lesssim h^{s-m}\fanshu[s]{\curl\pmb{v}}\quad 0\le m\le s,\label{estimateIh02}
\end{align}
if $\pmb{v}\in\pmb{H}_0(\mathrm{curl};\Omega)$ and $\curl\pmb{v}\notin\pmb{H}_0^1(\Omega)$
\begin{align}
\fanshu[0]{\pmb{v}-\pmb{I}_h^0\pmb{v}}&\lesssim h (\fanshu[1]{\pmb{v}}+\fanshu[1]{\curl\pmb{v}}),\label{estimateIh01A}\\
(\sum_{K\in\mesh}\fanshu[m,K]{\curl\pmb{v}-\curl \pmb{I}_h^0\pmb{v}}^2)^{1/2}&\lesssim h^{s-m}\fanshu[s]{\curl\pmb{v}}+h^{-m}\fanshu[0]{\curl\pmb{v}}\quad 0\le m\le s.\label{estimateIh02A}
\end{align}
\end{lemma}
\begin{proof}
We start with defining the projection $\pmb{\Pi}_h: \pmb{N}_h\rightarrow\pmb{U}_h$ locally as
\begin{align}
&\langle(\pmb{\Pi}_{h}\pmb{w}-\pmb{w})\cdot\pmb{t}_e,\kappa\rangle_e=0\qquad\qquad\qquad\text{for all }\kappa\in P^{k}(e) \text{ and edge }e\text{ of }K,\label{Pih1}\\
&\langle(\pmb{\Pi}_{h}\pmb{w}-\pmb{w})\times\pmb{n}_F,\pmb{\mu}\rangle_F=0\qquad\qquad\ \text{ for all }\pmb{\mu}\in \pmb{P}^{k-1}(F) \text{ and faces }F\text{ of }K,\label{Pih2}\\
&\langle(\curl\pmb{\Pi}_{h}\pmb{w})\times\pmb{n}_F-\mean{(\curl\pmb{w})\times\pmb{n}},\pmb{\chi}\rangle_F=0\ \text{ for all }\pmb{\chi}\in \pmb{P}^{k-1}(F) \text{ and faces }F\text{ of }K,\label{Pih3}\\
&(\pmb{\Pi}_{h}\pmb{w}-\pmb{w},\pmb{\rho})_K=0\qquad\qquad\qquad\quad\text{for all }\pmb{\rho}\in \pmb{P}^{k-2}(K).\label{Pih4}
\end{align}
According to \cite[Theorem 3.1]{zhang2022anew} and the definition of $\pmb{U}_h$, $\pmb{\Pi}_h$ is well-defined.

It is easy to check that for any simplex $K$ that contains $F$ as a face we have
\begin{align}\label{jumpequation}
\Big{|}\mean{\curl\pmb{w}}|_F-\curl(\pmb{w}|_K)\times\pmb{n}_F|_F\Big{|}=
\begin{cases}
\frac{1}{2}|\jump{\curl\pmb{w}\times\pmb{n}}|\quad &F\in\face^i,\\
0&F\in\face^b.
\end{cases}
\end{align}

Therefore by \eqref{Pih1}-\eqref{jumpequation} and a scaling argument, we obtain
\begin{align}\label{estimateL1I1}
\fanshu[0,K]{\pmb{\Pi}_h\pmb{w}-\pmb{w}}\lesssim h_K^{3/2}\sum_{F\subset\partial K\cap\face^i}\fanshu[0,F]{\jump{\curl\pmb{w}\times\pmb{n}}}.
\end{align}
We then set $\pmb{I}_h=\pmb{\Pi}_h\mathcal{J}_h^{N}$. If $s=0$, then by a scaling argument and \eqref{Jhdiagram} we have
\begin{align}
h_K^{3/2}\sum_{F\subset\partial K\cap\face^i}\fanshu[0,F]{\jump{(\curl\mathcal{J}_h^N\pmb{w})\times\pmb{n}}}&\lesssim h_K\fanshu[0,\mathcal{T}_K]{\curl\mathcal{J}_h^N\pmb{w}}\notag\\
&=h_K\fanshu[0,\mathcal{T}_K]{\mathcal{J}_h^M\curl\pmb{w}}.\label{estimateJhbound}
\end{align}
From \eqref{estimateL1I1} and \eqref{estimateJhbound} we get
\begin{align*}
\fanshu[0,K]{\pmb{v}-\pmb{I}_h\pmb{v}}&\le\fanshu[0,K]{\pmb{v}-\mathcal{J}_h^N\pmb{v}}+\fanshu[0,K]{\mathcal{J}_h^N\pmb{v}-\pmb{\Pi}_h\mathcal{J}_h^N\pmb{v}}\\
&\lesssim \fanshu[0,K]{\pmb{v}-\mathcal{J}_h^N\pmb{v}}+h_K^{3/2}\sum_{F\subset\partial K\cap\face^i}\fanshu[0,F]{\jump{\curl\mathcal{J}_h^N\pmb{v}\times\pmb{n}}}\\
&\lesssim\fanshu[0,K]{\pmb{v}-\mathcal{J}_h^N\pmb{v}}+h_K\fanshu[0,\mathcal{T}_K]{\mathcal{J}_h^M\curl\pmb{v}}.
\end{align*}
The above estimate and \eqref{JhBounded} lead to
\begin{align}
\fanshu[0]{\pmb{v}-\pmb{I}_h\pmb{v}}&\lesssim \fanshu[0]{\pmb{v}-\mathcal{J}_h^N\pmb{v}}+h\fanshu[0]{\mathcal{J}_h^M\curl\pmb{v}}\notag\\
&\lesssim\fanshu[0]{\pmb{v}}+\fanshu[0]{\curl\pmb{v}}.\label{estimateL2vIhs0}
\end{align}
Using the inverse estimate, \eqref{estimateL1I1} and \eqref{estimateJhbound} we have
\begin{align*}
\fanshu[0,K]{\curl(\pmb{v}-\pmb{I}_h\pmb{v})}&\le\fanshu[0,K]{\curl(\pmb{v}-\mathcal{J}_h^N\pmb{v})}+\fanshu[0,K]{\curl(\mathcal{J}_h^N\pmb{v}-\pmb{\Pi}_h\mathcal{J}_h^N\pmb{v})}\\
&\lesssim\fanshu[0,K]{\curl\pmb{v}-\mathcal{J}_h^M\curl\pmb{v}}+h_K^{-1}\fanshu[0,K]{\mathcal{J}_h^N\pmb{v}-\pmb{\Pi}_h\mathcal{J}_h^N\pmb{v}}\\
&\lesssim\fanshu[0,K]{\curl\pmb{v}-\mathcal{J}_h^M\curl\pmb{v}}+\fanshu[0,\mathcal{T}_K]{\mathcal{J}_h^M\curl\pmb{v}}.
\end{align*}
The above estimate and \eqref{JhBounded} yield
\begin{align}\label{curlvminusIhS0}
\fanshu[0]{\curl(\pmb{v}-\pmb{I}_h\pmb{v})}\lesssim\fanshu[0]{\curl\pmb{v}-\mathcal{J}_h^M\curl\pmb{v}}+\fanshu[0]{\mathcal{J}_h^M\curl\pmb{v}}\lesssim\fanshu[0]{\curl\pmb{v}}.
\end{align}

We now consider $s\ge1$. Note that $\jump{\curl\pmb{w}\times\pmb{n}}$ vanishes on $F\in\face^i$, if $\curl\pmb{w}\in \pmb{H}^1(\Omega)$. Then using the trace inequality \eqref{trace3} and \eqref{Jhdiagram} we derive
\begin{align}
h_K^{3/2}\sum_{F\subset\partial K\cap\face^i}&\fanshu[0,F]{\jump{(\curl\mathcal{J}_h^N\pmb{w})\times\pmb{n}}}=h_K^{3/2}\sum_{F\subset\partial K\cap\face^i}\fanshu[0,F]{\jump{(\curl(\mathcal{J}_h^N\pmb{w}-\pmb{w}))\times\pmb{n}}}\notag\\
&\lesssim h_K\fanshu[0,\mathcal{T}_K]{\curl(\pmb{w}-\mathcal{J}_h^N\pmb{w})}+h_K^2|\curl(\pmb{w}-\mathcal{J}_h^N\pmb{w})|_{1,\mathcal{T}_K}\notag\\
&=h_K\fanshu[0,\mathcal{T}_K]{\curl\pmb{w}-\mathcal{J}_h^M\curl\pmb{w}}+h_K^2|\curl\pmb{w}-\mathcal{J}_h^M\curl\pmb{w}|_{1,\mathcal{T}_K}.\label{estimateJhboundS1}
\end{align}
From \eqref{estimateL1I1} and \eqref{estimateJhboundS1} we obtain
\begin{align*}
\fanshu[0,K]{\pmb{v}&-\pmb{I}_h\pmb{v}}\le\fanshu[0,K]{\pmb{v}-\mathcal{J}_h^N\pmb{v}}+\fanshu[0,K]{\mathcal{J}_h^N\pmb{v}-\pmb{\Pi}_h\mathcal{J}_h^N\pmb{v}}\\
&\lesssim \fanshu[0,K]{\pmb{v}-\mathcal{J}_h^N\pmb{v}}+h_K^{3/2}\sum_{F\subset\partial K\cap\face^i}\fanshu[0,F]{\jump{\curl\mathcal{J}_h^N\pmb{v}\times\pmb{n}}}\\
&\lesssim\fanshu[0,K]{\pmb{v}-\mathcal{J}_h^N\pmb{v}}+h_K\fanshu[0,\mathcal{T}_K]{\curl\pmb{v}-\mathcal{J}_h^M\curl\pmb{v}}+h_K^2|\curl\pmb{v}-\mathcal{J}_h^M\curl\pmb{v}|_{1,\mathcal{T}_K}.
\end{align*}
By the above estimate, Lemma \ref{lemmaJhM} and \eqref{Jhinf} we have
\begin{align}
\fanshu[0]{\pmb{v}-\pmb{I}_h\pmb{v}}&\lesssim\fanshu[0]{\pmb{v}-\mathcal{J}_h^N\pmb{v}}+h\fanshu[0]{\curl\pmb{v}-\mathcal{J}_h^M\curl\pmb{v}}\notag\\
&+h^2(\sum_{K\in\mesh}|\curl\pmb{v}-\mathcal{J}_h^M\curl\pmb{v}|_{1,K}^2)^{1/2}\notag\\
&\lesssim h^s\fanshu[s]{\pmb{v}}+h^s\fanshu[s]{\curl\pmb{v}}.\label{estimateL2vIhS1}
\end{align}
Combing \eqref{estimateL2vIhs0} and \eqref{estimateL2vIhS1} we arrive at \eqref{estimateIh1}.

Using the inverse estimate, \eqref{estimateL1I1} and \eqref{estimateJhboundS1} we have
\begin{align*}
\fanshu[m,K]{\curl(\pmb{v}-\pmb{I}_h\pmb{v})}&\le\fanshu[m,K]{\curl(\pmb{v}-\mathcal{J}_h^N\pmb{v})}+\fanshu[m,K]{\curl(\mathcal{J}_h^N\pmb{v}-\pmb{\Pi}_h\mathcal{J}_h^N\pmb{v})}\\
&\lesssim\fanshu[m,K]{\curl\pmb{v}-\mathcal{J}_h^M\curl\pmb{v}}+h_K^{-m-1}\fanshu[0,K]{\mathcal{J}_h^N\pmb{v}-\pmb{\Pi}_h\mathcal{J}_h^N\pmb{v}}\\
&\lesssim\fanshu[m,K]{\curl\pmb{v}-\mathcal{J}_h^M\curl\pmb{v}}+h_K^{-m}\fanshu[0,\mathcal{T}_K]{\curl\pmb{v}-\mathcal{J}_h^M\curl\pmb{v}}\\
&+h_K^{1-m}|\curl\pmb{v}-\mathcal{J}_h^M\curl\pmb{v}|_{1,\mathcal{T}_K}.
\end{align*}
By the above estimate and Lemma \ref{lemmaJhM} we obtain
\begin{align}
(\sum_{K\in\mesh}&\fanshu[m,K]{\curl(\pmb{v}-\pmb{I}_h\pmb{v})}^2)^{1/2}\lesssim(\sum_{K\in\mesh}\fanshu[m,K]{\curl\pmb{v}-\mathcal{J}_h^M\curl\pmb{v}}^2)^{1/2}\notag\\
&+h^{-m}(\sum_{K\in\mesh}\fanshu[0,K]{\curl\pmb{v}-\mathcal{J}_h^M\curl\pmb{v}}^2)^{1/2}+h^{1-m}|(\sum_{K\in\mesh}\curl\pmb{v}-\mathcal{J}_h^M\curl\pmb{v}|_{1,\mathcal{T}_K}^2)^{1/2}\notag\\
&\lesssim h^{s-m}\fanshu[s]{\curl\pmb{v}}.\label{estimateCurlvIhS1}
\end{align}
Combing \eqref{curlvminusIhS0} and \eqref{estimateCurlvIhS1} we derive \eqref{estimateIh2}.

We now turn to the construction of $\pmb{I}_h^0$. To this end, we modify the construction of $\pmb{\Pi}_h$ by replacing \eqref{Pih3} with
\begin{equation}
\begin{aligned}
\langle(\curl\pmb{\Pi}_{h}\pmb{w})&\times\pmb{n}_F-\mean{(\curl\pmb{w})\times\pmb{n}},\pmb{\chi}\rangle_F=0\\ &\text{ for all }\pmb{\chi}\in \pmb{P}^{k-1}(F) \text{ and interior faces }F\text{ of }K,\\
\langle(\curl\pmb{\Pi}_{h}\pmb{w})\times\pmb{n}_F,\pmb{\chi}\rangle_F&=0\ \text{ for all }\pmb{\chi}\in \pmb{P}^{k-1}(F) \text{ and boundary faces }F\text{ of }K.
\end{aligned}
\end{equation}
Recalling \eqref{jumpequation} for interior faces $F$ and using the fact that
\begin{align*}
\langle(\curl(\pmb{\Pi}_{h}\pmb{w}-\pmb{w}))\times\pmb{n}_F,\pmb{\chi}\rangle_F&=-\langle(\curl\pmb{w})\times\pmb{n}_F,\pmb{\chi}\rangle_F\\ \text{ for all }\pmb{\chi}\in \pmb{P}^{k-1}(F) &\text{ and boundary faces }F\text{ of }K,
\end{align*}
we get
\begin{align}\label{estimateL1I3}
\fanshu[0,K]{\pmb{\Pi}_h\pmb{w}-\pmb{w}}\lesssim h_K^{3/2}\sum_{F\subset\partial K}\fanshu[0,F]{\jump{\curl\pmb{w}\times\pmb{n}}}.
\end{align}
We point out that, in contrast to \eqref{estimateL1I1}, the right-hand side of \eqref{estimateL1I3} may include boundary faces. If $\pmb{v}\in\pmb{H}_0(\mathrm{grad}\ \mathrm{curl};\Omega)$, then by the same argument as \eqref{estimateJhbound}and\eqref{estimateJhboundS1} we obtain
\begin{align*}
h_K^{3/2}\sum_{F\subset\partial K\cap\face}\fanshu[0,F]{\jump{(\curl\mathcal{J}_h^N\pmb{w})\times\pmb{n}}}&\lesssim h_K\fanshu[0,\mathcal{T}_K]{\mathcal{J}_h^M\curl\pmb{w}}, \text{ if }s=0,\\
h_K^{3/2}\sum_{F\subset\partial K\cap\face}\fanshu[0,F]{\jump{(\curl\mathcal{J}_h^N\pmb{w})\times\pmb{n}}}&\lesssim h_K\fanshu[0,\mathcal{T}_K]{\curl\pmb{w}-\mathcal{J}_h^M\curl\pmb{w}}\\
&+h_K^2|\curl\pmb{w}-\mathcal{J}_h^M\curl\pmb{w}|_{1,\mathcal{T}_K},\text{ if }s\ge1.
\end{align*}
Therefore, if $\pmb{v}\in\pmb{H}_0(\mathrm{grad}\ \mathrm{curl};\Omega)$, then from the same argument as in the proof of \eqref{estimateIh1} and \eqref{estimateIh2} we get \eqref{estimateIh01} and \eqref{estimateIh02}.

However, if $\pmb{v}\in\pmb{H}_0(\mathrm{curl};\Omega)$ and $\curl\pmb{v}\notin\pmb{H}_0^1(\Omega)$ then we can only get the estimate
\begin{align}
h_K^{3/2}\sum_{F\subset\partial K\cap\face^b}\fanshu[0,F]{\jump{(\curl\mathcal{J}_h^{N}\pmb{v})\times\pmb{n}}}\lesssim h_K\fanshu[0,\mathcal{T}_K]{\mathcal{J}_h^M\curl\pmb{v}}.
\end{align}
It then follows that
\begin{align*}
\fanshu[0,K]{\pmb{v}&-\pmb{I}_h^0\pmb{v}}\le\fanshu[0,K]{\pmb{v}-\mathcal{J}_h^N\pmb{v}}+\fanshu[0,K]{\mathcal{J}_h^N\pmb{v}-\pmb{\Pi}_h\mathcal{J}_h^N\pmb{v}}\\
&\lesssim \fanshu[0,K]{\pmb{v}-\mathcal{J}_h^N\pmb{v}}+h_K^{3/2}\sum_{F\subset\partial K}\fanshu[0,F]{\jump{\curl\mathcal{J}_h^N\pmb{v}\times\pmb{n}}}\\
&\lesssim \fanshu[0,K]{\pmb{v}-\mathcal{J}_h^N\pmb{v}}+h_K\fanshu[0,\mathcal{T}_K]{\mathcal{J}_h^M\curl\pmb{v}}.
\end{align*}
and
\begin{align*}
\fanshu[m,K]{\curl(\pmb{v}-\pmb{I}_h^0\pmb{v})}&\le \fanshu[m,K]{\curl(\pmb{v}-\mathcal{J}_h^{N}\pmb{v})}+\fanshu[m,K]{\curl(\mathcal{J}_h^{N}\pmb{v}-\pmb{\Pi}_h\mathcal{J}_h^{N}\pmb{v})}\\
&\lesssim\fanshu[m,K]{\curl(\pmb{v}-\mathcal{J}_h^{N}\pmb{v})}+h_K^{3/2-m-1}\sum_{F\subset\partial T}\fanshu[0,F]{\jump{(\curl\mathcal{J}_h^{N}\pmb{v})\times\pmb{n}}}\\
&\lesssim\fanshu[m,K]{\curl(\pmb{v}-\mathcal{J}_h^{N}\pmb{v})}+h_K^{-m}\fanshu[0,\mathcal{T}_K]{\mathcal{J}_h^M\curl\pmb{v}}.
\end{align*}
Therefore we have
\begin{align*}
\fanshu[0]{\pmb{v}-\pmb{I}_h^0\pmb{v}}&\lesssim \fanshu[0]{\pmb{v}-\mathcal{J}_h^N\pmb{v}}+h\fanshu[0]{\curl\pmb{v}}\notag\\
&\lesssim h(\fanshu[1]{\pmb{v}}+\fanshu[1]{\curl\pmb{v}})
\end{align*}
and
\begin{align*}
(\sum_{K\in\mesh}\fanshu[m,K]{\curl\pmb{v}-\curl \pmb{I}_h^0\pmb{v}}^2)^{1/2}&\lesssim
(\sum_{K\in\mesh}\fanshu[m,K]{\curl(\pmb{v}-\mathcal{J}_h^{N}\pmb{v})}^2)^{1/2}+h^{-m}\fanshu[0]{\mathcal{J}_h^M\curl\pmb{v}}\\
&\lesssim h^{s-m}\fanshu[s]{\curl\pmb{v}}+h^{-m}\fanshu[0]{\curl\pmb{v}}.
\end{align*}

\end{proof}

\section{The finite element method}\label{sectionApplication}
This section is devoted to introduce the finite element method for \eqref{SPquadcurlproblem} and analyze its convergence. To this end, we consider the following equivalent problem of \eqref{SPquadcurlproblem}
\begin{equation}\label{SPcurlLapcurlproblem}
\begin{aligned}
-\varepsilon^2\curl\Delta\curl\pmb{u}+(\nabla\times)^2\pmb{u}&=\pmb{f}\ \text{in }\Omega,\\
\nabla\cdot\pmb{u}&=0\ \text{in }\Omega,\\
\pmb{u}\times\pmb{n}&=0\ \text{on }\partial\Omega,\\
\nabla\times\pmb{u}&=0\ \text{on }\partial\Omega.
\end{aligned}
\end{equation}
The above problem can be obtained by the identity $\curl\curl\pmb{v}=-\Delta\pmb{v}+\nabla\nabla\cdot\pmb{v}$ and the fact that
\begin{align*}
\curl\pmb{u}\cdot\pmb{n}=(\pmb{n}\times\nabla)\cdot\pmb{u}=(\pmb{n}\times\nabla)\cdot(\pmb{n}\times\pmb{u}\times\pmb{n})=0\text{ on }\partial\Omega.
\end{align*}
Then the variational problem reads: Find $\pmb{u}\in\pmb{H}_0^1(\mathrm{grad}~\mathrm{curl};\Omega)$ and $p\in H_0^1(\Omega)$ such that
\begin{equation}
\begin{aligned}
A_{\varepsilon}(\pmb{u},\pmb{v})+(\pmb{v},\nabla p)&=(\pmb{f},\pmb{v}),\quad\forall\pmb{v}\in\pmb{H}_0^1(\mathrm{grad}~\mathrm{curl};\Omega),\notag\\
(\pmb{u},\nabla q)&=0,\quad\forall q\in H_0^1(\Omega),
\end{aligned}
\end{equation}
where
\begin{align}
A_{\varepsilon,h}(\pmb{v},\pmb{w})=\varepsilon^2a(\pmb{v},\pmb{w})+(\curl\pmb{v},\curl\pmb{w}),\label{definitionAepsh}
\end{align}
with
\begin{align}
a(\pmb{v},\pmb{w})=(\nabla\curl\pmb{v},\nabla\curl\pmb{w}).
\end{align}

\subsection{Analysis of imposing boundary conditions weakly}
To define the finite element method, we introduce the bilinear form
\begin{align}\label{definitionah}
a_h(\pmb{v},\pmb{w})&=\sum_{K\in\mesh}(\nabla\curl\pmb{v},\nabla\curl\pmb{w})_K-\sum_{F\in\face^b}\langle\frac{\partial}{\partial n}\curl\pmb{v},\curl\pmb{w}\rangle_F\notag\\
&-\sum_{F\in\edge^b}\langle\frac{\partial}{\partial n}\curl\pmb{w},\curl\pmb{v}\rangle_F+\sigma\sum_{F\in\face^b}h_F^{-1}\langle\curl\pmb{v},\curl\pmb{w}\rangle_F,
\end{align}
where $\sigma$ is a positive penalization parameter, and $\frac{\partial}{\partial n}=\pmb{n}_F\cdot\nabla$.

The finite element method then reads: Find $\pmb{u}_h\in\pmb{U}_h$ and $p_h\in W_{h,0}$ such that
\begin{equation}
\begin{aligned}
A_{\varepsilon,h}(\pmb{u}_h,\pmb{v}_h)+(\pmb{v}_h,\nabla p_h)&=(\pmb{f},\pmb{v}_h),\quad\forall\pmb{v}_h\in\pmb{U}_h,\\
(\pmb{u}_h,\nabla q_h)&=0,\quad\forall q_h\in W_{h,0},
\end{aligned}\label{discreteSPquad}
\end{equation}
where
\begin{align}
A_{\varepsilon,h}(\pmb{v}_h,\pmb{w}_h)=\varepsilon^2a_h(\pmb{v}_h,\pmb{w}_h)+(\curl\pmb{v}_h,\curl\pmb{w}_h).
\end{align}
\begin{remark}
Inserting $\pmb{v}_h=\nabla q_h,\ \forall q_h\in W_{h,0}$ in the first equation of \eqref{discreteSPquad}, we deduce $p_h=0$ from the fact $\divv\pmb{f}=0$. Therefore, the solution $\pmb{u}_h\in\pmb{U}_{h}$ satisfies
\begin{align}
A_{\varepsilon,h}(\pmb{u}_h,\pmb{v}_h)=(\pmb{f},\pmb{v}_h),\quad\forall\pmb{v}_h\in\pmb{U}_{h}.\label{equationAehfh}
\end{align}
\end{remark}

We define the following norms associated with problem \eqref{discreteSPquad}
\begin{align}
\fanshu[gc,h]{\pmb{v}}^2&=\sum_{K\in\mesh}\fanshu[0,K]{\nabla\curl\pmb{v}}^2+\sum_{F\in\face^b}h_F^{-1}\fanshu[0,F]{\curl\pmb{v}}^2,\\
\fanshu[\varepsilon,h]{\pmb{v}}^2&=\varepsilon^2\fanshu[gc,h]{\pmb{v}}^2+\fanshu[0]{\curl\pmb{v}}^2+\fanshu[0]{\pmb{v}}^2.\label{discreteEnergynorm}
\end{align}

The next lemma establishes the coercivity of $A_{\varepsilon,h}(\cdot,\cdot)$ on the kernel $\pmb{Z}_h$, where
\begin{align*}
\pmb{Z}_h = \{\pmb{v}_h\in\pmb{U}_{h}:(\pmb{v}_h,\nabla q_h)=0,\ \forall q_h\in W_{h,0} \}.
\end{align*}

\begin{lemma}\label{lemmaCoerAeh}
There exists a constant $\sigma_0>0$ depending only on the shape regularity of $\mesh$, such that for $\sigma\ge\sigma_0$, there holds
\begin{align}
\fanshu[\varepsilon,h]{\pmb{w}}^2\lesssim A_{\varepsilon,h}(\pmb{w},\pmb{w})\quad\forall\pmb{w}\in\pmb{Z}_h.
\end{align}
\end{lemma}
\begin{proof}
According to \cite[Lemma4.1]{zhang2022anew}, the following discrete Friedrichs inequality holds
\begin{align*}
\fanshu[0]{\pmb{w}}\lesssim\fanshu[0]{\curl\pmb{w}},\ \forall\pmb{w}\in\pmb{Z}_{h}.
\end{align*}
As a result, we have
\begin{align*}
\fanshu[\varepsilon,h]{\pmb{w}}^2\lesssim 2(\fanshu[0]{\curl\pmb{w}}^2+\varepsilon^2\fanshu[gc,h]{\pmb{w}}^2).
\end{align*}
Therefore it suffices to show that the bilinear form $a_h(\cdot,\cdot)$ is coercive
with respect to the norm $\fanshu[gc,h]{\pmb{w}}$. By \eqref{definitionah}, we obtain
\begin{align}\label{ahww}
a_h(\pmb{w},\pmb{w})&=\sum_{K\in\mesh}|\curl\pmb{w}|_{1,K}^2-2\sum_{F\in\face^b}\langle\frac{\partial}{\partial n}\curl\pmb{w},\curl\pmb{w}\rangle_F\notag\\
&+\sigma\sum_{F\in\face^b}h_F^{-1}\fanshu[0,F]{\curl\pmb{w}}^2.
\end{align}
By the Cauchy-Schwarz inequality and a standard scaling argument, there exists a constant $C>0$ depending only on the shape regularity of the mesh such that
\begin{align*}
2&\sum_{F\in\face^b}\langle\frac{\partial}{\partial n}\curl\pmb{w},\curl\pmb{w}\rangle_F\\
&\le2\left(\sum_{F\in\face^b}h_F\fanshu[0,F]{\frac{\partial}{\partial n}\curl\pmb{w}}^2\right)^{1/2}\left(\sum_{F\in\face^b}h_F^{-1}\fanshu[0,F]{\curl\pmb{w}}^2\right)^{1/2}\\
&\le C\left(\sum_{F\in\face^b}h_F\banfan[1,\mathcal{T}_F]{\curl\pmb{w}}^2\right)^{1/2}\left(\sum_{F\in\face^b}h_F^{-1}\fanshu[0,F]{\curl\pmb{w}}^2\right)^{1/2}\\
&\le \frac{1}{2}\sum_{T\in\mesh}\banfan[1,K]{\curl\pmb{w}}^2+\frac{C^2}{2}\sum_{F\in\face^b}h_F^{-1}\fanshu[0,F]{\curl\pmb{w}}.
\end{align*}
Using the above estimate and \eqref{ahww} we get
\begin{align*}
a_h(\pmb{w},\pmb{w})\ge\frac{1}{2}\sum_{K\in\mesh}\banfan[1,K]{\curl\pmb{w}}^2+\left(\sigma-\frac{C^2}{2}\right)\sum_{F\in\face^b}h_F^{-1}\fanshu[0,F]{\curl\pmb{w}}^2.
\end{align*}
Choosing $\sigma_0=\frac{1}{2}(C^2+1)$, we derive the desired estimate.
\end{proof}

The next lemma shows that the finite element method \eqref{discreteSPquad} is well-posed.
\begin{lemma}
We have the discrete stability
\begin{align}\label{discreteStability}
&\fanshu[\varepsilon,h]{\widetilde{\pmb{u}}_h}+|\widetilde{p}_h|_1\lesssim\sup_{(\pmb{v}_h,q_h)\in\pmb{U}_{h}\times W_{h,0}}\frac{A_{\varepsilon,h}(\widetilde{\pmb{u}}_h,\pmb{v}_h)+b(\pmb{v}_h,\widetilde{p}_h)+b(\widetilde{\pmb{u}}_h,q_h)}{\fanshu[\varepsilon,h]{\pmb{v}_h}+|q_h|_1}
\end{align}
for any $(\widetilde{\pmb{u}}_h,\widetilde{p}_h)\in\pmb{U}_{h}\times W_{h,0}$.
\end{lemma}
\begin{proof}
Since $\nabla q_h\in \pmb{U}_{h,0},\ \forall q_h\in W_{h,0}$, we have the inf-sup condition
\begin{align*}
\sup_{\pmb{0}\neq\pmb{v}_h\in\pmb{U}_h}\frac{b(\pmb{v}_h,q_h)}{\fanshu[\varepsilon,h]{\pmb{v}_h}}\ge \frac{b(\nabla q_h,q_h)}{\fanshu[\varepsilon,h]{\nabla q_h}} \gtrsim |q_h|_1 \quad\forall q_h\in W_{h,0}.
\end{align*}
As a result, the discrete stability follows from Lemma \ref{lemmaCoerAeh}, the above inf-sup condition and the Babu\v{s}ka-Brezzi theory \cite{boffi2013mixed}.
\end{proof}

Before proceeding further, we need some technical lemmas. The following lemma gives the interpolation error in the norm defined by \eqref{discreteEnergynorm}.
\begin{lemma}\label{lemmaerrepshnorm}
Let $\pmb{u}$ be the solution to \eqref{SPcurlLapcurlproblem} satisfying $\pmb{u}\in \pmb{H}^s(\Omega)$ and $\curl\pmb{u}\in \pmb{H}^s(\Omega)$ with $1\le s\le k+1$, and let $\pmb{I}_h$ be the operator given in Lemma \ref{lemmaIh}. Then we have
\begin{align}
\fanshu[\varepsilon,h]{\pmb{u}-\pmb{I}_h\pmb{u}}&\lesssim(h^s+\varepsilon h^{s-1})(\fanshu[s]{\pmb{u}}+\fanshu[s]{\curl\pmb{u}}),\label{errepshnorm1}\\
\fanshu[\varepsilon,h]{\pmb{u}-\pmb{I}_h\pmb{u}}&\lesssim h^{1/2}\fanshu[0]{\pmb{f}}.\label{errepshnorm2}
\end{align}
Moreover, if the solution to the reduced problem satisfies $\curl\overline{\pmb{u}}\in \pmb{H}^m(\Omega)$ with $1\le m\le k+1$
\begin{align}
\fanshu[\varepsilon,h]{\pmb{u}-\pmb{I}_h\pmb{u}}\lesssim\varepsilon^{1/2}\fanshu[0]{\pmb{f}}+h^m(\fanshu[m]{\overline{\pmb{u}}}+\fanshu[m]{\curl\overline{\pmb{u}}}).\label{errepshnorm3}
\end{align}
\end{lemma}
\begin{proof}
Using \eqref{estimateIh2} and \eqref{regularitySP1} we obtain
\begin{align}
\varepsilon^2\sum_{K\in\mesh}\fanshu[0,K]{\nabla\curl(\pmb{u}-\pmb{I}_h\pmb{u})}^2&\lesssim \varepsilon^2h^{2s-2}\fanshu[s]{\curl\pmb{u}}^2,\label{errgcnorm1}\\
\varepsilon^2\sum_{K\in\mesh}\fanshu[0,K]{\nabla\curl(\pmb{u}-\pmb{I}_h\pmb{u})}^2&\lesssim \varepsilon^2h\fanshu[2]{\pmb{u}}\fanshu[2]{\curl\pmb{u}}\lesssim h\fanshu[0]{\pmb{f}}^2,\label{errgcnorm2}\\
\varepsilon^2\sum_{K\in\mesh}\fanshu[0,K]{\nabla\curl(\pmb{u}-\pmb{I}_h\pmb{u})}^2&\lesssim\varepsilon^2\fanshu[2]{\pmb{u}}^2\lesssim \varepsilon\fanshu[0]{\pmb{f}}^2.\label{errgcnorm3}
\end{align}

Similarly, we use \eqref{estimateIh2}, \eqref{regularitySP1} and \eqref{trace3} to get
\begin{align}
\varepsilon^2&\sum_{F\in\face^b}h_F^{-1}\fanshu[0,F]{\curl(\pmb{u}-\pmb{I}_h\pmb{u})}^2\notag\\
&\lesssim\varepsilon^2\sum_{K\in\mesh}\left(h_K^{-2}\fanshu[0,K]{\curl(\pmb{u}-\pmb{I}_h\pmb{u})}+\fanshu[0,K]{\nabla\curl(\pmb{u}-\pmb{I}_h\pmb{u})}\right)\notag\\
&\lesssim \varepsilon^2h^{2s-2}\fanshu[s]{\curl\pmb{u}}^2,\label{errfacecurl1}
\end{align}
\begin{align}
\varepsilon^2&\sum_{F\in\face^b}h_F^{-1}\fanshu[0,F]{\curl(\pmb{u}-\pmb{I}_h\pmb{u})}^2\notag\\
&\lesssim\varepsilon^2\sum_{K\in\mesh}\left(h_K^{-2}\fanshu[0,K]{\curl(\pmb{u}-\pmb{I}_h\pmb{u})}^2+\fanshu[0,K]{\nabla\curl(\pmb{u}-\pmb{I}_h\pmb{u})}^2\right)\notag\\
&\lesssim \varepsilon^2h\fanshu[2]{\pmb{u}}\fanshu[2]{\curl\pmb{u}}\lesssim h\fanshu[0]{\pmb{f}}^2,\label{errfacecurl2}
\end{align}
and
\begin{align}
\varepsilon^2&\sum_{F\in\face^b}h_F^{-1}\fanshu[0,F]{\curl(\pmb{u}-\pmb{I}_h\pmb{u})}^2\notag\\
&\lesssim\varepsilon^2\sum_{K\in\mesh}\left(h_K^{-2}\fanshu[0,K]{\curl(\pmb{u}-\pmb{I}_h\pmb{u})}+\fanshu[0,K]{\nabla\curl(\pmb{u}-\pmb{I}_h\pmb{u})}\right)\notag\\
&\lesssim\varepsilon^2\fanshu[2]{\pmb{u}}^2\lesssim \varepsilon\fanshu[0]{\pmb{f}}^2.\label{errfacecurl3}
\end{align}

Applying \eqref{estimateIh2}, \eqref{regularitySP1}, \eqref{regularitySP2} and \eqref{regularityreduced}, we have
\begin{align}
\fanshu[0]{\curl(\pmb{u}-\pmb{I}_h\pmb{u})}^2&\lesssim h^{2s}\fanshu[s]{\curl\pmb{u}}^2,\label{errcurlnorm1}\\
\fanshu[0]{\curl(\pmb{u}-\pmb{I}_h\pmb{u})}^2&\lesssim \fanshu[0]{\curl(\pmb{u}-\overline{\pmb{u}})-\curl\pmb{I}_h(\pmb{u}-\overline{\pmb{u}})}^2+\fanshu[0]{\curl(\overline{\pmb{u}}-\pmb{I}_h\overline{\pmb{u}})}^2\notag\\
&\lesssim h\fanshu[0]{\curl(\pmb{u}-\overline{\pmb{u}})}\fanshu[1]{\curl(\pmb{u}-\overline{\pmb{u}})}+h^2\fanshu[1]{\curl\overline{\pmb{u}}}^2\notag\\
&\lesssim h\fanshu[0]{\pmb{f}}^2,\label{errcurlnorm2}
\end{align}
and
\begin{align}
\fanshu[0]{\curl(\pmb{u}-\pmb{I}_h\pmb{u})}^2&\lesssim \fanshu[0]{\curl(\pmb{u}-\overline{\pmb{u}})-\curl\pmb{I}_h(\pmb{u}-\overline{\pmb{u}})}^2+\fanshu[0]{\curl(\overline{\pmb{u}}-\pmb{I}_h\overline{\pmb{u}})}^2\notag\\
&\lesssim\fanshu[0]{\curl(\pmb{u}-\overline{\pmb{u}})-\curl\pmb{I}_h(\pmb{u}-\overline{\pmb{u}})}^2+h^{2m}\fanshu[m]{\curl\overline{\pmb{u}}}^2\notag\\
&\lesssim\varepsilon\fanshu[0]{\pmb{f}}^2+h^{2m}\fanshu[m]{\curl\overline{\pmb{u}}}^2.\label{errcurlnorm3}
\end{align}

Using \eqref{estimateIh1} and \eqref{regularitySP1}-\eqref{regularityreduced}, we have
\begin{align}
\fanshu[0]{\pmb{u}-\pmb{I}_h}&\lesssim h^{2s}(\fanshu[s]{\pmb{u}}+\fanshu[s]{\curl\pmb{u}})\label{errL2norm1}\\
\fanshu[0]{\pmb{u}-\pmb{I}_h}&\lesssim \fanshu[0]{\pmb{u}-\overline{\pmb{u}}-\pmb{I}_h(\pmb{u}-\overline{\pmb{u}})}^2+\fanshu[0]{\overline{\pmb{u}}-\pmb{I}_h\overline{\pmb{u}}}^2\notag\\
&\lesssim h(\fanshu[0]{\pmb{u}-\overline{\pmb{u}}}+\fanshu[0]{\curl(\pmb{u}-\overline{\pmb{u}})})(\fanshu[1]{\pmb{u}-\overline{\pmb{u}}}+\fanshu[1]{\curl(\pmb{u}-\overline{\pmb{u}})})\notag\\
&+h^2(\fanshu[1]{\overline{\pmb{u}}}+\fanshu[1]{\curl\overline{\pmb{u}}})^2\notag\\
&\lesssim h\fanshu[0]{\pmb{f}}^2\label{errL2norm2}
\end{align}
and
\begin{align}
\fanshu[0]{\pmb{u}-\pmb{I}_h}&\lesssim \fanshu[0]{\pmb{u}-\overline{\pmb{u}}-\pmb{I}_h(\pmb{u}-\overline{\pmb{u}})}^2+\fanshu[0]{\overline{\pmb{u}}-\pmb{I}_h\overline{\pmb{u}}}^2\notag\\
&\lesssim(\fanshu[0]{\pmb{u}-\overline{\pmb{u}}}+\fanshu[0]{\curl(\pmb{u}-\overline{\pmb{u}})})^2+h^{2m}(\fanshu[m]{\overline{\pmb{u}}}+\fanshu[m]{\curl\overline{\pmb{u}}})^2\notag\\
&\lesssim \varepsilon\fanshu[0]{\pmb{f}}^2++\fanshu[0]{\curl(\pmb{u}-\overline{\pmb{u}})})^2+h^{2m}(\fanshu[m]{\overline{\pmb{u}}}+\fanshu[m]{\curl\overline{\pmb{u}}})^2.\label{errL2norm3}
\end{align}

Finally, the estimate \eqref{errepshnorm1} can be obtained by \eqref{discreteEnergynorm}, \eqref{errgcnorm1}, \eqref{errfacecurl1}, \eqref{errcurlnorm1} and \eqref{errL2norm1}; the estimate \eqref{errepshnorm2} can be obtained by \eqref{discreteEnergynorm}, \eqref{errgcnorm2}, \eqref{errfacecurl2}, \eqref{errcurlnorm2} and \eqref{errL2norm2}; and the estimate \eqref{errepshnorm3} can be obtained by \eqref{discreteEnergynorm}, \eqref{errgcnorm3}, \eqref{errfacecurl3}, \eqref{errcurlnorm3} and \eqref{errL2norm3}.
\end{proof}

The next lemma provides bounds for the interpolation error in the bilinear form restricted to the subspace.
\begin{lemma}\label{lemmaAeh}
Under the assumption of Lemma \ref{lemmaerrepshnorm}, for any $\pmb{w}\in\pmb{U}_h$ we have
\begin{align}
A_{\varepsilon,h}(\pmb{u}-\pmb{I}_h\pmb{u},\pmb{w})&\lesssim(1+\sigma)(\varepsilon h^{s-1}+h^s)(\fanshu[s]{\pmb{u}}+\fanshu[s]{\curl\pmb{u}})\fanshu[\varepsilon,h]{\pmb{w}},\label{errAepsh1}\\
A_{\varepsilon,h}(\pmb{u}-\pmb{I}_h\pmb{u},\pmb{w})&\lesssim(1+\sigma)h^{1/2}\fanshu[0]{\pmb{f}}\fanshu[\varepsilon,h]{\pmb{w}},\label{errAepsh2}\\
A_{\varepsilon,h}(\pmb{u}-\pmb{I}_h\pmb{u},\pmb{w})&\lesssim(1+\sigma)(\varepsilon^{1/2} \fanshu[0]{\pmb{f}}+h^m(\fanshu[m]{\overline{\pmb{u}}}+\fanshu[m]{\curl\overline{\pmb{u}}}))\fanshu[\varepsilon,h]{\pmb{w}}.\label{errAepsh3}
\end{align}
\end{lemma}

\begin{proof}
It follows from \eqref{definitionAepsh} that
\begin{align}
A_{\varepsilon,h}&(\pmb{u}-\pmb{I}_h\pmb{u},\pmb{w})\notag\\
&=\varepsilon^2\sum_{K\in\mesh}(\nabla\curl(\pmb{u}-\pmb{I}_h\pmb{u}),\nabla\curl\pmb{w})_K+(\curl(\pmb{u}-\pmb{I}_h\pmb{u}),\curl\pmb{w})\notag\\
&-\varepsilon^2\sum_{F\in\face^b}\langle\frac{\partial}{\partial n}\curl(\pmb{u}-\pmb{I}_h\pmb{u}),\curl\pmb{w}\rangle_F-\varepsilon^2\sum_{F\in\face^b}\langle\curl(\pmb{u}-\pmb{I}_h\pmb{u}),\frac{\partial}{\partial n}\curl\pmb{w}\rangle_F\notag\\
&+\varepsilon^2\sigma\sum_{F\in\face^b}h_F^{-1}\langle\curl(\pmb{u}-\pmb{I}_h\pmb{u}),\curl\pmb{w}\rangle_F\notag\\
&:=J_1+J_2+J_3+J_4+J_5.\label{defJ12345}
\end{align}
Using the Cauchy-Schwarz inequality, \eqref{trace4} and \eqref{definitionAepsh}, we have
\begin{align*}
J_1,J_2,J_4,J_5\lesssim(1+\sigma)\fanshu[\varepsilon,h]{\pmb{u}-\pmb{I}_h\pmb{u}}\fanshu[\varepsilon,h]{\pmb{w}}.
\end{align*}
Then by \eqref{errepshnorm1}-\eqref{errepshnorm3} we obtain
\begin{align}
J_1,J_2,J_4,J_5&\lesssim(1+\sigma)(\varepsilon h^{s-1}+h^s)(\fanshu[s]{\pmb{u}}+\fanshu[s]{\curl\pmb{u}})\fanshu[\varepsilon,h]{\pmb{w}},\label{J1245a}\\
J_1,J_2,J_4,J_5&\lesssim(1+\sigma)h^{1/2}\fanshu[0]{\pmb{f}}\fanshu[\varepsilon,h]{\pmb{w}},\label{J1245b}
\end{align}
and
\begin{align}
J_1,J_2,J_4,J_5\lesssim(1+\sigma)(\varepsilon^{1/2} \fanshu[0]{\pmb{f}}+h^m(\fanshu[m]{\overline{\pmb{u}}}+\fanshu[m]{\curl\overline{\pmb{u}}}))\fanshu[\varepsilon,h]{\pmb{w}}.\label{J1245c}
\end{align}

We now estimate $J_3$. By the Cauchy-Schwarz inequality, the trace inequality \eqref{trace3}, \eqref{definitionAepsh} and \eqref{estimateIh2}, we get
\begin{align}
J_3&\lesssim\varepsilon^2\left(\sum_{F\in\face^b}h_F\fanshu[0,F]{\frac{\partial}{\partial n}\curl(\pmb{u}-\pmb{I}_h\pmb{u})}^2\right)^{1/2}\left(\sum_{F\in\face^b}h_F^{-1}\fanshu[0,F]{\curl\pmb{w}}^2\right)^{1/2}\notag\\
&\lesssim\varepsilon\left(\sum_{K\in\mesh}\fanshu[0,K]{\nabla(\curl\pmb{u}-\curl\pmb{I}_h\pmb{u})}^2+h_K^2\fanshu[1,K]{\nabla(\curl\pmb{u}-\curl\pmb{I}_h\pmb{u})}^2 \right)^{1/2}\fanshu[\varepsilon,h]{\pmb{w}}\notag\\
&\lesssim\varepsilon h^{s-1}(\fanshu[s]{\pmb{u}}+\fanshu[s]{\curl\pmb{u}})\fanshu[\varepsilon,h]{\pmb{w}}.\label{estimateJ31}
\end{align}
Moreover, by the trace inequality \eqref{trace3}, \eqref{estimateIh2}, \eqref{trace4} and \eqref{regularitySP1}, we obtain
\begin{align}
J_3&\lesssim\varepsilon^2\left(\sum_{K\in\mesh}h_K^{-2}\fanshu[0,K]{\nabla(\curl\pmb{u}-\curl\pmb{I}_h\pmb{u})}^2+\fanshu[1,K]{\nabla(\curl\pmb{u}-\curl\pmb{I}_h\pmb{u})}^2 \right)^{1/2}\fanshu[0]{\curl\pmb{w}}\notag\\
&\lesssim\varepsilon^2\fanshu[2]{\curl\pmb{u}}\fanshu[0]{\curl\pmb{w}}\lesssim\varepsilon^{1/2}\fanshu[0]{\pmb{f}}\fanshu[\varepsilon,h]{\pmb{w}}.\label{estimateJ32}
\end{align}
We emphasize that if $\varepsilon\le h$, then the above estimate means that
\begin{align}
J_3\lesssim h^{1/2}\fanshu[0]{\pmb{f}}\fanshu[\varepsilon,h]{\pmb{w}}.\label{estimateJ33}
\end{align}
In addition, if $h\le\varepsilon$, then by \eqref{estimateJ31} with $s=2$ and \eqref{regularitySP1} we have
\begin{align}
J_3\lesssim\varepsilon h\fanshu[2]{\curl\pmb{u}}\fanshu[\varepsilon,h]{\pmb{w}}\lesssim h^{1/2}\fanshu[0]{\pmb{f}}\fanshu[\varepsilon,h]{\pmb{w}}.\label{estimateJ34}
\end{align}

Finally, the estimate \eqref{errAepsh1} follows from \eqref{defJ12345}, \eqref{J1245a} and \eqref{estimateJ31}; the estimate \eqref{errAepsh2} follows from \eqref{defJ12345}, \eqref{J1245b}, \eqref{estimateJ33} and \eqref{estimateJ34}; and the estimate \eqref{errAepsh3} follows from \eqref{defJ12345} and \eqref{J1245c}.
\end{proof}

With the help of previous lemmas we now can prove the main result of this section.
\begin{theorem}\label{theoremUUhWeak}
Suppose that $\pmb{u}\in \pmb{H}^s(\Omega)$ and $\curl\pmb{u}\in \pmb{H}^s(\Omega)$ with $1\le s\le k+1$. Then we have
\begin{align}
\fanshu[\varepsilon,h]{\pmb{u}-\pmb{u}_h}&\lesssim(h^s+\varepsilon h^{s-1})(\fanshu[s]{\pmb{u}}+\fanshu[s]{\curl\pmb{u}}),\label{erruuh1}\\
\fanshu[\varepsilon,h]{\pmb{u}-\pmb{u}_h}&\lesssim h^{1/2}\fanshu[0]{\pmb{f}}.\label{erruuh2}
\end{align}
Moreover, if the solution to the reduced problem satisfies $\curl\overline{\pmb{u}}\in \pmb{H}^m(\Omega)$ with $1\le m\le k+1$
\begin{align}
\fanshu[\varepsilon,h]{\pmb{u}-\pmb{u}_h}\lesssim(\varepsilon^{1/2}\fanshu[0]{\pmb{f}}+h^m(\fanshu[m]{\overline{\pmb{u}}}+\fanshu[m]{\curl\overline{\pmb{u}}})).\label{erruuh3}
\end{align}
\end{theorem}
\begin{remark}
It is worth pointing out that for $\varepsilon=o(h)$ the estimate \eqref{erruuh3} gives better order approximation than the uniform order $h^{1/2}$.
\end{remark}
\begin{proof}
Applying \eqref{discreteStability} with $\widetilde{\pmb{u}}_h=\pmb{I}_h\pmb{u}-\pmb{u}_h$ and $\widetilde{p}_h=0$, recalling \eqref{equationAehfh} and $\pmb{u}_h\in\pmb{Z}_h$, we have
\begin{align}
\fanshu[\varepsilon,h]{\pmb{I}_h\pmb{u}-\pmb{u}_h}&\lesssim\sup_{(\pmb{v}_h,q_h)\in\pmb{U}_{h}\times W_{h,0}}\frac{A_{\varepsilon,h}(\pmb{I}_h\pmb{u}-\pmb{u}_h,\pmb{v}_h)+b(\pmb{I}_h\pmb{u}-\pmb{u}_h,q_h)}{\fanshu[\varepsilon,h]{\pmb{v}_h}+|q_h|_1}\notag\\
&=\sup_{(\pmb{v}_h,q_h)\in\pmb{U}_{h}\times W_{h,0}}\frac{A_{\varepsilon,h}(\pmb{I}_h\pmb{u},\pmb{v}_h)-(\pmb{f},\pmb{v}_h)+b(\pmb{I}_h\pmb{u}-\pmb{u},q_h)}{\fanshu[\varepsilon,h]{\pmb{v}_h}+|q_h|_1}\notag\\
&\lesssim\fanshu[0]{\pmb{u}-\pmb{I}_h\pmb{u}}+\sup_{\pmb{v}_h\in\pmb{U}_{h}}\frac{A_{\varepsilon,h}(\pmb{I}_h\pmb{u}-\pmb{u},\pmb{v}_h)+E_h(\pmb{u},\pmb{v}_h)}{\fanshu[\varepsilon,h]{\pmb{v}_h}},\label{errIhuuh}
\end{align}
where the consistency error is defined as
\begin{align}
E_h(\pmb{u},\pmb{v}_h)=A_{\varepsilon,h}(\pmb{u},\pmb{v}_h)-(\pmb{f},\pmb{v}_h).\label{DefEhuv}
\end{align}
Multiplying both sides of the first equation of \eqref{SPcurlLapcurlproblem} by $\pmb{v}_h\in\pmb{U}_h$ and integrating by parts, we get
\begin{align*}
(\pmb{f},\pmb{v}_h)&=\varepsilon^2\sum_{K\in\mesh}(\nabla\curl\pmb{u},\pmb{v}_h)_K-\varepsilon^2\sum_{F\in\face^b}\langle\frac{\partial}{\partial n}\curl\pmb{u},\curl\pmb{v}_h\rangle_F\\
&-\varepsilon^2\sum_{F\in\face^i}\langle\frac{\partial}{\partial n}\curl\pmb{u},\jump{\curl\pmb{v}_h}\rangle_F+(\curl\pmb{u},\curl\pmb{v}_h).
\end{align*}
Substituting the above equation into \eqref{DefEhuv} yields
\begin{align}
E_h(\pmb{u},\pmb{v}_h)=\varepsilon^2\sum_{F\in\face^i}\langle\frac{\partial}{\partial n}\curl\pmb{u},\jump{\curl\pmb{v}_h}\rangle_F.
\end{align}
According to \cite[Lemma 3.2]{zhang2022anew}, there holds for any $\pmb{\chi}\in\pmb{P}^{k-1}(F)$,
\begin{align}
\langle\frac{\partial}{\partial n}\curl\pmb{u},\jump{\curl\pmb{v}_h}\rangle_F=\langle\frac{\partial}{\partial n}\curl\pmb{u}-\pmb{\chi},\jump{\curl\pmb{v}_h}\rangle_F.
\end{align}
Hence setting $\pmb{\chi}=\frac{\partial}{\partial n}\curl\mathcal{J}_h^N\pmb{u}$, we have
\begin{align}
\langle\frac{\partial}{\partial n}\curl\pmb{u},\jump{\curl\pmb{v}_h}\rangle_F&=\langle\frac{\partial}{\partial n}\curl(\pmb{u}-\mathcal{J}_h^N\pmb{u}),\jump{\curl\pmb{v}_h}\rangle_F\notag\\
&\lesssim\fanshu[0,F]{\mean{\frac{\partial}{\partial n}(\curl\pmb{u}-\mathcal{J}_h^M\curl\pmb{u})}}\fanshu[0,F]{\jump{\curl\pmb{v}_h}}.\label{estimate1}
\end{align}
We now estimate each of the factors on the right of the above inequality separately. Noting that the average of $\jump{\curl\pmb{v}_h}$ vanishes on $F$, we have by Poincare's inequality that
\begin{align*}
\fanshu[0,F]{\jump{\curl\pmb{v}_h}}\lesssim h_F\fanshu[0,F]{\nabla_F(\jump{\curl\pmb{v}_h})},
\end{align*}
where $\nabla_F$ is the surface gradient on $F$. Using \eqref{trace4} we get
\begin{align}
\fanshu[0,F]{\jump{\curl\pmb{v}_h}}\lesssim h_F^{1/2}\banfan[1,\mathcal{T}_F]{\curl\pmb{v}_h}.\label{estimate2}
\end{align}
Using \eqref{trace1} we have
\begin{align}
\fanshu[0,F]{\mean{&\frac{\partial}{\partial n}(\curl\pmb{u}-\mathcal{J}_h^M\curl\pmb{u})}}\lesssim h_F^{-1/2}\banfan[1,\mathcal{T}_F]{\curl\pmb{u}-\mathcal{J}_h^M\curl\pmb{u}}\notag\\
&+\banfan[1,\mathcal{T}_F]{\curl\pmb{u}-\mathcal{J}_h^M\curl\pmb{u}}^{1/2}\banfan[2,\mathcal{T}_F]{\curl\pmb{u}-\mathcal{J}_h^M\curl\pmb{u}}^{1/2}\label{estimate3}
\end{align}
We then have by \eqref{estimate1}, \eqref{estimate2}, \eqref{estimate3} and Lemma \ref{lemmaJhM} that
\begin{align}
E_h(\pmb{u},\pmb{v}_h)\lesssim\varepsilon h^{s-1}\fanshu[s]{\curl\pmb{u}}\fanshu[\varepsilon,h]{\pmb{v}_h}.\label{estimateEuvh1}
\end{align}
The error estimate \eqref{erruuh1} then follows from \eqref{errIhuuh}, \eqref{errAepsh1}, \eqref{estimateEuvh1}, the triangle inequality and \eqref{errepshnorm1}.

Setting $s=1,2$ in \eqref{estimateEuvh1}, we get
\begin{align}
E_h(\pmb{u},\pmb{v}_h)\lesssim\varepsilon\fanshu[1]{\curl\pmb{u}}\fanshu[\varepsilon,h]{\pmb{v}_h}\label{estimateEuvh1a}
\end{align}
and
\begin{align}
E_h(\pmb{u},\pmb{v}_h)\lesssim\varepsilon h\fanshu[2]{\curl\pmb{u}}\fanshu[\varepsilon,h]{\pmb{v}_h},\label{estimateEuvh1b}
\end{align}
respectively.
As a result, we obtain by \eqref{estimateEuvh1a}, \eqref{estimateEuvh1b} and \eqref{regularitySP1} that
\begin{align}
E_h(\pmb{u},\pmb{v}_h)=E_h(\pmb{u},\pmb{v}_h)^{1/2}E_h(\pmb{u},\pmb{v}_h)^{1/2}\lesssim h^{1/2} \fanshu[0]{\pmb{f}}\fanshu[\varepsilon,h]{\pmb{v}_h}.\label{estimateEuvh2}
\end{align}
The error estimate \eqref{erruuh2} then follows from \eqref{errIhuuh}, \eqref{errAepsh2}, \eqref{estimateEuvh2}, \eqref{errepshnorm2}.

Finally, from \eqref{estimateEuvh1a} and \eqref{regularitySP1}, we have
\begin{align}
E_h(\pmb{u},\pmb{v}_h)\lesssim \varepsilon^{1/2} \fanshu[0]{\pmb{f}}\fanshu[\varepsilon,h]{\pmb{v}_h}.\label{estimateEuvh3}
\end{align}
The last estimate \eqref{erruuh3} then follows from \eqref{errIhuuh}, \eqref{errAepsh3}, \eqref{estimateEuvh3}, \eqref{errepshnorm3}.
\end{proof}

\subsection{Remarks on imposing boundary conditions strongly}
Similar to other non-conforming methods for the fourth order elliptic singular perturbation problem \eqref{SPbiharmonic} (e.g. \cite{Wang2006modified,Nilssen2001arobust}), we can define a finite element method for \eqref{SPcurlLapcurlproblem} that impose the second boundary condition strongly in the finite element space. Although this approach is perhaps more natural than the previous one, we shall see that the convergence properties of the previous method is better than this one.

In this case, we define the method as finding $\pmb{u}_h^0\in\pmb{U}_{h,0}$ and $p_h^0\in W_{h,0}$ such that
\begin{equation}
\begin{aligned}
A_{\varepsilon,h}^0(\pmb{u}_h^0,\pmb{v}_h)+(\pmb{v}_h,\nabla p_h^0)&=(\pmb{f},\pmb{v}_h),\quad\forall\pmb{v}_h\in\pmb{U}_h^0,\\
(\pmb{u}_h^0,\nabla q_h)&=0,\quad\forall q_h\in W_{h,0},
\end{aligned}\label{discreteSPquadStrong}
\end{equation}
where
\begin{align}
A_{\varepsilon,h}^0(\pmb{v}_h,\pmb{w}_h)=\varepsilon^2\sum_{K\in\mesh}(\nabla\curl\pmb{v},\nabla\curl\pmb{w})_K+(\curl\pmb{v}_h,\curl\pmb{w}_h).
\end{align}
Define the following norm
\begin{align}
\fanshu[A]{\pmb{v}_h}^2:=A_{\varepsilon,h}^0(\pmb{v}_h,\pmb{v}_h)+\fanshu[0]{\pmb{v}_h}^2.\label{defNormA}
\end{align}
Using \cite[Remark 4.1]{zhang2022anew}, one can easily see that $A_{\varepsilon,h}^0(\cdot,\cdot)$ is coercive on $\widetilde{\pmb{Z}}_h$, where
\begin{align*}
\widetilde{\pmb{Z}}_h = \{\pmb{v}_h\in\pmb{U}_{h,0}:(\pmb{v}_h,\nabla q_h)=0,\ \forall q_h\in W_{h,0} \},
\end{align*}
and the following inf-sup condition holds
\begin{align*}
\sup_{\pmb{0}\neq\pmb{v}_h\in\pmb{U}_{h,0}}\frac{b(\pmb{v}_h,q_h)}{\fanshu[A]{\pmb{v}_h}}\ge \frac{b(\nabla q_h,q_h)}{\fanshu[A]{\nabla q_h}} \gtrsim |q_h|_1 \quad\forall q_h\in W_{h,0}.
\end{align*}
Furthermore, we have the following discrete stability
\begin{align}\label{discreteStabilityStrong}
&\fanshu[A]{\widetilde{\pmb{u}}_h}+|\widetilde{p}_h|_1\lesssim\sup_{(\pmb{v}_h,q_h)\in\pmb{U}_{h,0}\times W_{h,0}}\frac{A_{\varepsilon,h}^0(\widetilde{\pmb{u}}_h,\pmb{v}_h)+b(\pmb{v}_h,\widetilde{p}_h)+b(\widetilde{\pmb{u}}_h,q_h)}{\fanshu[A]{\pmb{v}_h}+|q_h|_1}
\end{align}
for any $(\widetilde{\pmb{u}}_h,\widetilde{p}_h)\in\pmb{U}_{h,0}\times W_{h,0}$. Therefore \eqref{discreteSPquadStrong} is well-posed.

Next, we give a result similar to Lemma \ref{lemmaerrepshnorm}.
\begin{lemma}\label{lemmaerrepshnormIh0}
Let $\pmb{u}$ be the solution to \eqref{SPcurlLapcurlproblem} satisfying $\pmb{u}\in \pmb{H}^s(\Omega)$ and $\curl\pmb{u}\in \pmb{H}^s(\Omega)$ with $1\le s\le k+1$, and let $\pmb{I}_h^0$ be the operator given in Lemma \ref{lemmaIh}. Then we have
\begin{align}
\fanshu[A]{\pmb{u}-\pmb{I}_h^0\pmb{u}}&\lesssim(h^s+\varepsilon h^{s-1})(\fanshu[s]{\pmb{u}}+\fanshu[s]{\curl\pmb{u}}),\label{errepshnorm1Ih0}\\
\fanshu[A]{\pmb{u}-\pmb{I}_h^0\pmb{u}}&\lesssim h^{1/2}\fanshu[0]{\pmb{f}}.\label{errepshnorm2Ih0}
\end{align}
\end{lemma}
\begin{proof}
The inequality \eqref{errepshnorm1Ih0} can be proved by the method analogous to that used in Lemma \ref{lemmaIh}.

We now focus on \eqref{errepshnorm2Ih0}. An argument similar to the one used in \eqref{errgcnorm2} shows that
\begin{align*}
\varepsilon^2\sum_{K\in\mesh}\fanshu[0,K]{\nabla\curl(\pmb{u}-\pmb{I}_h^0\pmb{u})}^2&\lesssim \varepsilon^2h\fanshu[2]{\pmb{u}}\fanshu[2]{\curl\pmb{u}}\lesssim h\fanshu[0]{\pmb{f}}^2.
\end{align*}
We now turn to estimate $\fanshu[0]{\curl(\pmb{u}-\pmb{I}_h^0\pmb{u})}$. Using the triangle inequality, we have
\begin{align*}
\fanshu[0]{\curl(\pmb{u}-\pmb{I}_h^0\pmb{u})}\le\fanshu[0]{\curl(\pmb{u}-\pmb{I}_h\pmb{u})}+\fanshu[0]{\curl(\pmb{I}_h\pmb{u}-\pmb{I}_h^0\pmb{u})}.
\end{align*}
Note that the first term on the right-hand side of the above inequality has been estimated in \eqref{errepshnorm2}. By the definitions of $\pmb{I}_h$ and $\pmb{I}_h^0$, we can see that $\pmb{I}_h\pmb{u}|_K = \pmb{I}_h^0\pmb{u}|_K$ for $K\in\mesh$ satisfying $\partial K\cap\partial\Omega=\emptyset$. Therefore, by a scaling argument, we have
\begin{align*}
\fanshu[0,K]{\curl(\pmb{I}_h\pmb{u}-\pmb{I}_h^0\pmb{u})}\lesssim\sum_{F\in\face^b}h_F\fanshu[0,F]{\pmb{n}\times\curl(\mathcal{J}_h^N\pmb{u})}.
\end{align*}
Using the above inequality and recalling that $\curl\pmb{u}|_{\partial\Omega}=0$, we get
\begin{align*}
\sum_{K\in\mesh}\fanshu[0,K]{\curl(\pmb{I}_h\pmb{u}-\pmb{I}_h^0\pmb{u})}^2&\lesssim\sum_{F\in\face^b}h_F\fanshu[0,F]{\pmb{n}\times\curl(\mathcal{J}_h^N\pmb{u})}^2\\
&=\sum_{F\in\face^b}h_F\fanshu[0,F]{\pmb{n}\times\curl(\mathcal{J}_h^N\pmb{u}-\pmb{u})}^2.
\end{align*}
Adding and subtracting proper terms, we have
\begin{align*}
\sum_{F\in\face^b}&h_F\fanshu[0,F]{\pmb{n}\times\curl(\mathcal{J}_h^N\pmb{u}-\pmb{u})}^2\\
&\lesssim\sum_{F\in\face^b}h_F\fanshu[0,F]{\pmb{n}\times\curl(\mathcal{J}_h^N(\pmb{u}-\pmb{u}^0)-(\pmb{u}-\pmb{u}^0))}^2\\
&+\sum_{F\in\face^b}h_F\fanshu[0,F]{\pmb{n}\times\curl(\mathcal{J}_h^N\pmb{u}^0-\pmb{u}^0)}^2.
\end{align*}
Using the trace inequality \eqref{trace1}, \eqref{lemmaJhM}, \eqref{regularitySP1} and \eqref{regularitySP2}, we derive
\begin{align*}
\sum_{F\in\face^b}&h_F\fanshu[0,F]{\pmb{n}\times\curl(\mathcal{J}_h^N(\pmb{u}-\pmb{u}^0)-(\pmb{u}-\pmb{u}^0))}^2\\
&\lesssim h\fanshu[0]{\curl(\pmb{u}-\pmb{u}^0)}\fanshu[1]{\curl(\pmb{u}-\pmb{u}^0)}\\
&\lesssim h\fanshu[0]{\pmb{f}}^2,
\end{align*}
and
\begin{align*}
\sum_{F\in\face^b}h_F\fanshu[0,F]{\pmb{n}\times\curl(\mathcal{J}_h^N\pmb{u}^0-\pmb{u}^0)}^2\lesssim h\fanshu[0]{\pmb{f}}^2.
\end{align*}
Finally, combing the above inequalities we arrive at \eqref{errepshnorm2Ih0}.
\end{proof}

\begin{theorem}\label{theoremUUhStrong}
Let $\pmb{u}\in \pmb{H}^s(\Omega)$ and $\curl\pmb{u}\in \pmb{H}^s(\Omega)$ with $1\le s\le k+1$. Then we have
\begin{align}
\fanshu[A]{\pmb{u}-\pmb{u}_h^0}&\lesssim(h^s+\varepsilon h^{s-1})(\fanshu[s]{\pmb{u}}+\fanshu[s]{\curl\pmb{u}}),\label{erruuh1Strong}\\
\fanshu[A]{\pmb{u}-\pmb{u}_h^0}&\lesssim h^{1/2}\fanshu[0]{\pmb{f}}.\label{erruuh2Strong}
\end{align}
\end{theorem}
\begin{proof}
We only sketch the proof, because the proof of Theorem \ref{theoremUUhStrong} is similar to that of Theorem \ref{theoremUUhWeak}.

By \eqref{defNormA}, \eqref{discreteSPquadStrong} and \eqref{discreteStabilityStrong} we have
\begin{align}
\fanshu[A]{\pmb{I}_h^0\pmb{u}-\pmb{u}_h^0}\lesssim\fanshu[0]{\pmb{u}-\pmb{I}_h^0\pmb{u}}+\sup_{\pmb{v}_h\in\pmb{U}_{h,0}}\frac{A_{\varepsilon,h}^0(\pmb{I}_h^0\pmb{u}-\pmb{u},\pmb{v}_h)+E_h^0(\pmb{u},\pmb{v}_h)}{\fanshu[A]{\pmb{v}_h}},\label{errIh0uuh}
\end{align}
where the consistency error is defined as
\begin{align}
E_h^0(\pmb{u},\pmb{v}_h)=A_{\varepsilon,h}^0(\pmb{u},\pmb{v}_h)-(\pmb{f},\pmb{v}_h).\label{DefEh0uv}
\end{align}

By the proof of Lemma \ref{lemmaAeh}, and using \eqref{estimateIh01}-\eqref{estimateIh02A}, we obtain
\begin{align}
A_{\varepsilon,h}^0(\pmb{u}-\pmb{I}_h^0\pmb{u},\pmb{v}_h)&\lesssim(\varepsilon h^{s-1}+h^s)(\fanshu[s]{\pmb{u}}+\fanshu[s]{\curl\pmb{u}})\fanshu[A]{\pmb{v}_h},\label{errAepsh01}\\
A_{\varepsilon,h}^0(\pmb{u}-\pmb{I}_h^0\pmb{u},\pmb{v}_h)&\lesssim h^{1/2}\fanshu[0]{\pmb{f}}\fanshu[A]{\pmb{v}_h}.\label{errAepsh02}
\end{align}
Integrating by parts leads to
\begin{align*}
E_h^0(\pmb{u},\pmb{v}_h)=\varepsilon^2\sum_{F\in\face}\langle\frac{\partial}{\partial n}\curl\pmb{u},\jump{\curl\pmb{v}_h}\rangle_F.
\end{align*}
Using the same arguments as in the proof of Theorem \ref{theoremUUhWeak}, we get
\begin{align}
E_h^0(\pmb{u},\pmb{v}_h)&\lesssim\varepsilon h^{s-1}\fanshu[s]{\curl\pmb{u}}\fanshu[A]{\pmb{v}_h}.\label{estimateEuvh01}\\
E_h^0(\pmb{u},\pmb{v}_h)&\lesssim h^{1/2} \fanshu[0]{\pmb{f}}\fanshu[A]{\pmb{v}_h}.\label{estimateEuvh02}
\end{align}

Finally, the estimate \eqref{erruuh1Strong} then follows from \eqref{errIh0uuh}, \eqref{errAepsh01}, \eqref{estimateEuvh01} and \eqref{errepshnorm1Ih0}, and the estimate \eqref{erruuh2Strong} then follows from \eqref{errIh0uuh}, \eqref{errAepsh02}, \eqref{estimateEuvh02} and \eqref{errepshnorm2Ih0}.
\end{proof}
\begin{remark}
Note that the error $\pmb{u}-\pmb{u}_h^0$ does not satisfy an estimate of the form \eqref{erruuh3}. This is because we can not obtain an estimate of the interpolation error of the form \eqref{errepshnorm3}. By looking closely at the proof of Lemma \ref{lemmaerrepshnorm}, we find a problem occurs when we try to estimate the term $\fanshu[0]{\curl(\pmb{u}-\pmb{I}_h^0\pmb{u})}$ due to the result \eqref{estimateIh02A}. More precisely, in contrast to \eqref{errcurlnorm3}, we have by \eqref{estimateIh02A} and \eqref{regularitySP1},
\begin{align*}
\fanshu[0]{\curl(\pmb{u}-\pmb{I}_h^0\pmb{u})}&\lesssim \fanshu[0]{\curl(\pmb{u}-\overline{\pmb{u}})-\curl\pmb{I}_h^0(\pmb{u}-\overline{\pmb{u}})}+\fanshu[0]{\curl(\overline{\pmb{u}}-\pmb{I}_h^0\overline{\pmb{u}})}\notag\\
&\lesssim\fanshu[0]{\curl(\pmb{u}-\overline{\pmb{u}})}+\fanshu[0]{\curl(\overline{\pmb{u}}-\pmb{I}_h^0\overline{\pmb{u}})}\notag\\
&\lesssim\varepsilon^{1/2}\fanshu[0]{\pmb{f}}+\fanshu[0]{\curl(\overline{\pmb{u}}-\pmb{I}_h^0\overline{\pmb{u}})}.
\end{align*}
However, from \eqref{estimateIh02A} we can only derive
\begin{align}
\fanshu[0]{\curl(\overline{\pmb{u}}-\pmb{I}_h^0\overline{\pmb{u}})}\lesssim\fanshu[0]{\curl\overline{\pmb{u}}},
\end{align}
which does not yield anything.

The above discussion illustrate that the advantages of imposing boundary conditions weakly. Particularly, if $\varepsilon\ll h$, then from \eqref{erruuh3} we will observe convergence rates of $O(h^{m})$ in the energy norm, which is better than $O(h^{1/2})$.
\end{remark}

\section{Numerical results}\label{sectionnumerical}
In this section we provide some numerical examples to demonstrate the convergence rates predicted in Theorem \ref{theoremUUhWeak} and Theorem \ref{theoremUUhStrong} and illustrate the advantages of weakly imposed the second boundary conditions in certain situations. We also compare our method with total nonconforming methods introduced in \cite{zheng2011nonconforming,huang2022nonconforing}. Our algorithms are implemented by the iFEM package \cite{chen2008ifem}. For simplicity, we only consider the case $k=1$.

\subsection{Example 1}
We consider the quad-curl problem \eqref{SPcurlLapcurlproblem} on the unit cube $\Omega=(0,1)^3$. The exact solution is
\begin{align}
\pmb{u}=\left(
          \begin{array}{c}
            \sin(\pi x)^3\sin(\pi y)^2\sin(\pi z)^2\cos(\pi y)\cos(\pi z) \\
            \sin(\pi y)^3\sin(\pi z)^2\sin(\pi x)^2\cos(\pi z)\cos(\pi x) \\
          -2\sin(\pi z)^3\sin(\pi x)^2\sin(\pi y)^2\cos(\pi x)\cos(\pi y) \\
          \end{array}
        \right),
\end{align}
with the following data
\begin{align*}
\pmb{f}=-\varepsilon^2\curl\Delta\curl\pmb{u}+(\nabla\times)^2\pmb{u},\quad \sigma=50.
\end{align*}
Although $\pmb{f}$ depends on $\varepsilon$, this does not affect convergence rates in this case due to the fact that $\fanshu[2]{\pmb{u}}$ and $\fanshu[2]{\curl\pmb{u}}$ are independent of $\varepsilon$. The mesh $\mesh$ is obtained by dividing the unit cube into $N^3$ small cubes and then partition each small cube into $6$ tetrahedra. Given $\pmb{u}$ and $\pmb{u}_h$, the following error quantities are considered in this example:
\begin{align*}
E_{L^2}(\pmb{u},\pmb{u}_h)&=\fanshu[0]{\pmb{u}-\pmb{u}_h}/\fanshu[0]{\pmb{u}},\\
E_{\mathrm{curl}}(\pmb{u},\pmb{u}_h)&=\fanshu[0]{\nabla\times(\pmb{u}-\pmb{u}_h)}/\fanshu[0]{\curl\pmb{u}},\\
E_{\mathrm{g}\mathrm{c}}(\pmb{u},\pmb{u}_h)&=\fanshu[0]{\nabla_h\curl{(\pmb{u}-\pmb{u}_h)}}/\fanshu[0]{\gcurl\pmb{u}},\\
E_{\varepsilon,h}(\pmb{u},\pmb{u}_h)&=(\varepsilon^2\fanshu[0]{\nabla_h\curl{(\pmb{u}-\pmb{u}_h)}}^2+\fanshu[0]{\nabla\times(\pmb{u}-\pmb{u}_h)}^2+\fanshu[0]{\pmb{u}-\pmb{u}_h}^2)^{1/2}/\fanshu[\varepsilon,h]{\pmb{u}},\\
E_{A}(\pmb{u},\pmb{u}_h)&=(\varepsilon^2\fanshu[0]{\nabla_h\curl{(\pmb{u}-\pmb{u}_h)}}^2+\fanshu[0]{\nabla\times(\pmb{u}-\pmb{u}_h)}^2+\fanshu[0]{\pmb{u}-\pmb{u}_h}^2)^{1/2}/\fanshu[A]{\pmb{u}}.
\end{align*}

\begin{table}[!htbp]\centering
\caption{Example 1: Numerical results by the nonconforming method \eqref{discreteSPquad}\label{tableexample1weak}}
\scalebox{1}[1]{
\begin{tabular}{*{10}{c@{\extracolsep{5pt}}@{\extracolsep{5pt}}}}\hline
$\varepsilon$&$h$ &$E_{L^2}(\pmb{u},\pmb{u}_h)$&  rate & $E_{\mathrm{curl}}(\pmb{u},\pmb{u}_h)$&  rate  & $E_{\mathrm{g}\mathrm{c}}(\pmb{u},\pmb{u}_h)$ &  rate& $E_{\varepsilon,h}(\pmb{u},\pmb{u}_h)$ &  rate\\
\hline
  & 0.2165& 1.734e-01&      & 1.966e-01&      & 4.659e-01&      & 4.822e-01&      \\
  & 0.1732& 1.179e-01& 1.73& 1.361e-01& 1.65& 3.864e-01& 0.84& 3.994e-01& 0.84\\
  & 0.1443& 8.438e-02& 1.83& 9.876e-02& 1.76& 3.288e-01& 0.89& 3.398e-01& 0.89\\
1 & 0.1237& 6.295e-02& 1.90& 7.448e-02& 1.83& 2.856e-01& 0.91& 2.954e-01& 0.91\\
  & 0.1083& 4.854e-02& 1.95& 5.796e-02& 1.88& 2.522e-01& 0.93& 2.612e-01& 0.92\\
  & 0.0962& 3.843e-02& 1.98& 4.627e-02& 1.91& 2.256e-01& 0.94& 2.341e-01& 0.93\\
  & 0.0866& 3.111e-02& 2.01& 3.774e-02& 1.93& 2.041e-01& 0.95& 2.121e-01& 0.93\\
\hline
  & 0.2165& 5.987e-02&      & 8.934e-02&      & 5.535e-01&      & 1.129e-01&      \\
  & 0.1732& 3.900e-02& 1.92& 5.943e-02& 1.83& 4.327e-01& 1.10& 8.013e-02& 1.54\\
  & 0.1443& 2.733e-02& 1.95& 4.230e-02& 1.87& 3.545e-01& 1.09& 6.098e-02& 1.50\\
$10^{-2}$ & 0.1237& 2.018e-02& 1.97& 3.163e-02& 1.89& 3.006e-01& 1.07& 4.884e-02& 1.44\\
  & 0.1083& 1.550e-02& 1.98& 2.455e-02& 1.90& 2.613e-01& 1.05& 4.061e-02& 1.38\\
  & 0.0962& 1.227e-02& 1.98& 1.961e-02& 1.90& 2.314e-01& 1.03& 3.475e-02& 1.33\\
  & 0.0866& 9.953e-03& 1.99& 1.604e-02& 1.91& 2.078e-01& 1.02& 3.038e-02& 1.28\\
\hline
  & 0.2165& 5.777e-02&      & 8.459e-02&      & 7.187e-01&      & 8.440e-02&      \\
  & 0.1732& 3.781e-02& 1.90& 5.546e-02& 1.89& 5.702e-01& 1.04& 5.534e-02& 1.89\\
  & 0.1443& 2.661e-02& 1.93& 3.901e-02& 1.93& 4.696e-01& 1.06& 3.892e-02& 1.93\\
$10^{-5}$ & 0.1237& 1.972e-02& 1.94& 2.887e-02& 1.95& 3.984e-01& 1.07& 2.881e-02& 1.95\\
  & 0.1083& 1.519e-02& 1.96& 2.221e-02& 1.97& 3.457e-01& 1.06& 2.216e-02& 1.97\\
  & 0.0962& 1.205e-02& 1.97& 1.760e-02& 1.97& 3.053e-01& 1.05& 1.756e-02& 1.97\\
  & 0.0866& 9.788e-03& 1.97& 1.429e-02& 1.98& 2.734e-01& 1.05& 1.426e-02& 1.98\\
\hline
\end{tabular}}
\end{table}

\begin{table}[!htbp]\centering
\caption{Example 1: Numerical results by the nonconforming method \eqref{discreteSPquadStrong}\label{tableexample1strong}}
\scalebox{1}[1]{
\begin{tabular}{*{10}{c@{\extracolsep{5pt}}@{\extracolsep{5pt}}}}\hline
$\varepsilon$&$h$ &$E_{L^2}(\pmb{u},\pmb{u}_h^0)$&  rate & $E_{\mathrm{curl}}(\pmb{u},\pmb{u}_h^0)$&  rate  & $E_{\mathrm{g}\mathrm{c}}(\pmb{u},\pmb{u}_h^0)$ &  rate& $E_{A}(\pmb{u},\pmb{u}_h^0)$ &  rate\\
\hline
  & 0.2165& 1.752e-01&      & 1.963e-01&      & 4.674e-01&      & 4.661e-01&      \\
  & 0.1732& 1.203e-01& 1.68& 1.366e-01& 1.63& 3.866e-01& 0.85& 3.855e-01& 0.85\\
  & 0.1443& 8.709e-02& 1.77& 9.965e-02& 1.73& 3.285e-01& 0.89& 3.275e-01& 0.89\\
1 & 0.1237& 6.571e-02& 1.83& 7.557e-02& 1.79& 2.850e-01& 0.92& 2.841e-01& 0.92\\
  & 0.1083& 5.122e-02& 1.87& 5.912e-02& 1.84& 2.514e-01& 0.94& 2.506e-01& 0.94\\
  & 0.0962& 4.099e-02& 1.89& 4.743e-02& 1.87& 2.247e-01& 0.95& 2.240e-01& 0.95\\
  & 0.0866& 3.351e-02& 1.91& 3.884e-02& 1.89& 2.031e-01& 0.96& 2.024e-01& 0.96\\
\hline
  & 0.2165& 6.055e-02&      & 9.042e-02&      & 5.678e-01&      & 1.126e-01&      \\
  & 0.1732& 3.934e-02& 1.93& 6.001e-02& 1.84& 4.394e-01& 1.15& 7.953e-02& 1.56\\
  & 0.1443& 2.752e-02& 1.96& 4.261e-02& 1.88& 3.578e-01& 1.13& 6.027e-02& 1.52\\
$10^{-2}$ & 0.1237& 2.029e-02& 1.98& 3.178e-02& 1.90& 3.022e-01& 1.10& 4.808e-02& 1.47\\
  & 0.1083& 1.556e-02& 1.99& 2.461e-02& 1.92& 2.620e-01& 1.07& 3.983e-02& 1.41\\
  & 0.0962& 1.230e-02& 1.99& 1.961e-02& 1.93& 2.315e-01& 1.05& 3.395e-02& 1.36\\
  & 0.0866& 9.970e-03& 2.00& 1.600e-02& 1.93& 2.076e-01& 1.04& 2.957e-02& 1.31\\
\hline
  & 0.2165& 5.796e-02&      & 8.535e-02&      & 7.499e-01&      & 8.515e-02&      \\
  & 0.1732& 3.789e-02& 1.91& 5.585e-02& 1.90& 5.903e-01& 1.07& 5.573e-02& 1.90\\
  & 0.1443& 2.665e-02& 1.93& 3.924e-02& 1.94& 4.835e-01& 1.09& 3.915e-02& 1.94\\
$10^{-5}$ & 0.1237& 1.974e-02& 1.95& 2.901e-02& 1.96& 4.086e-01& 1.09& 2.895e-02& 1.96\\
  & 0.1083& 1.520e-02& 1.96& 2.230e-02& 1.97& 3.535e-01& 1.08& 2.225e-02& 1.97\\
  & 0.0962& 1.206e-02& 1.97& 1.767e-02& 1.98& 3.115e-01& 1.07& 1.763e-02& 1.98\\
  & 0.0866& 9.792e-03& 1.97& 1.434e-02& 1.98& 2.784e-01& 1.07& 1.431e-02& 1.98\\
\hline
\end{tabular}}
\end{table}

\begin{table}[!htbp]\centering
\caption{Example 1: Numerical results by the nonconforming method \cite{zheng2011nonconforming}\label{tableexample1TNC}}
\scalebox{1}[1]{
\begin{tabular}{*{10}{c@{\extracolsep{5pt}}@{\extracolsep{5pt}}}}\hline
$\varepsilon$&$h$ &$E_{L^2}(\pmb{u},\hat{\pmb{u}}_h)$&  rate & $E_{\mathrm{curl}}(\pmb{u},\hat{\pmb{u}}_h)$&  rate  & $E_{\mathrm{g}\mathrm{c}}(\pmb{u},\hat{\pmb{u}}_h)$ &  rate& $E_{A}(\pmb{u},\hat{\pmb{u}}_h)$ &  rate\\
\hline
 &0.2165& 3.119e-01&      & 2.215e-01&      & 5.625e-01&      & 5.608e-01&      \\
 &0.1732& 2.053e-01& 1.87& 1.460e-01& 1.87& 4.584e-01& 0.92& 4.570e-01& 0.92\\
 &0.1443& 1.450e-01& 1.91& 1.031e-01& 1.91& 3.861e-01& 0.94& 3.849e-01& 0.94\\
1&0.1237& 1.076e-01& 1.93& 7.657e-02& 1.93& 3.331e-01& 0.96& 3.320e-01& 0.96\\
 &0.1083& 8.301e-02& 1.95& 5.904e-02& 1.95& 2.927e-01& 0.97& 2.918e-01& 0.97\\
 &0.0962& 6.592e-02& 1.96& 4.688e-02& 1.96& 2.610e-01& 0.97& 2.601e-01& 0.97\\
 &0.0866& 5.359e-02& 1.96& 3.811e-02& 1.97& 2.354e-01& 0.98& 2.346e-01& 0.98\\
\hline
 &0.2165& 1.407e+00&      & 9.889e-01&      & 3.275e+00&      & 1.062e+00&      \\
 &0.1732& 1.340e+00& 0.22& 9.452e-01& 0.20& 3.856e+00& -0.73& 1.050e+00& 0.05\\
 &0.1443& 1.253e+00& 0.37& 8.864e-01& 0.35& 4.285e+00& -0.58& 1.023e+00& 0.14\\
$10^{-2}$&0.1237& 1.159e+00& 0.51& 8.220e-01& 0.49& 4.586e+00& -0.44& 9.878e-01& 0.23\\
 &0.1083& 1.064e+00& 0.64& 7.574e-01& 0.61& 4.781e+00& -0.31& 9.491e-01& 0.30\\
 &0.0962& 9.737e-01& 0.76& 6.952e-01& 0.73& 4.894e+00& -0.20& 9.092e-01& 0.36\\
 &0.0866& 8.891e-01& 0.86& 6.369e-01& 0.83& 4.943e+00& -0.09& 8.697e-01& 0.42\\
\hline
 &0.2165& 1.604e+00&      & 1.133e+00&      & 3.903e+00&      & 1.138e+00&      \\
 &0.1732& 1.645e+00& -0.11& 1.164e+00& -0.12& 4.995e+00& -1.11& 1.169e+00& -0.12\\
 &0.1443& 1.670e+00& -0.08& 1.182e+00& -0.08& 6.076e+00& -1.07& 1.187e+00& -0.08\\
$10^{-5}$&0.1237& 1.686e+00& -0.06& 1.194e+00& -0.06& 7.149e+00& -1.06& 1.199e+00& -0.06\\
 &0.1083& 1.696e+00& -0.05& 1.201e+00& -0.05& 8.216e+00& -1.04& 1.206e+00& -0.05\\
 &0.0962& 1.703e+00& -0.04& 1.207e+00& -0.04& 9.280e+00& -1.03& 1.212e+00& -0.04\\
 &0.0866& 1.709e+00& -0.03& 1.210e+00& -0.03& 1.034e+01& -1.03& 1.215e+00& -0.03\\
\hline
\end{tabular}}
\end{table}

Table \ref{tableexample1weak} reports the convergence results for finite element method \eqref{discreteSPquad} for $\varepsilon=1$, $\varepsilon=10^{-2}$ and $\varepsilon=10^{-5}$, respectively. We see that the asymptotic convergence rate in energy norm is optimal, which confirms the theoretical analysis in Theorem \ref{theoremUUhWeak}. Since the exact solution has no layers, we also observe the optimal rates of convergence in other norms.

Table \ref{tableexample1strong} and \ref{tableexample1TNC} reports the convergence results for finite element method \eqref{discreteSPquadStrong} and total nonconforming method introduced in \cite{zheng2011nonconforming} for $\varepsilon=1$, $\varepsilon=10^{-2}$ and $\varepsilon=10^{-5}$, respectively. Here $\hat{\pmb{u}}_h$ is the solution of \eqref{discreteSPquadStrong} with the nonconforming space introduced in \cite{zheng2011nonconforming}. Again, since $\fanshu[2]{\pmb{u}}$ and $\fanshu[2]{\curl\pmb{u}}$ are independent of $\varepsilon$ for this example, we observe that the convergence rates in all norms for different values of $\varepsilon$ are optimal, which is consistent with Theorem \ref{theoremUUhStrong}. In contrast, we find that when $\varepsilon=1$, the convergence rates of total nonconforming method are optimal, while the convergence deteriorates as $\varepsilon$ approaches zero.

Although our method has more local degrees of freedom than that of \cite{zheng2011nonconforming} (28 for our method and 20 for \cite{zheng2011nonconforming}), our method is robust with respect to the perturbation parameter $\varepsilon$, while \cite{zheng2011nonconforming} is not robust with respect to $\varepsilon$. Therefore our method has a huge advantage in solving the singularly perturbed quad curl problem.

\subsection{Example 2}
In this example, we show that in some situations it is very advantageous to impose the second boundary condition weakly. This time we take the exact solution to the reduced problem \eqref{reducedproblem} to be
\begin{align*}
\overline{\pmb{u}}=\left(
          \begin{array}{c}
             0 \\
             - \frac{x^2\, y^2\, z\, {\left(x - 1\right)}^3\, {\left(y - 1\right)}^3\, {\left(z - 1\right)}^3}{4} - \frac{3\, x^2\, y^2\, z^2\, {\left(x - 1\right)}^3\, {\left(y - 1\right)}^3\, {\left(z - 1\right)}^2}{8} \\
          \frac{x^2\, y\, z^2\, {\left(x - 1\right)}^3\, {\left(y - 1\right)}^3\, {\left(z - 1\right)}^3}{4} + \frac{3\, x^2\, y^2\, z^2\, {\left(x - 1\right)}^3\, {\left(y - 1\right)}^2\, {\left(z - 1\right)}^3}{8} \\
          \end{array}
        \right),
\end{align*}
with the following data
\begin{align*}
\pmb{f}=(\nabla\times)^2\overline{\pmb{u}}&,\quad \sigma=20,\quad \varepsilon=10^{-6},\quad\Omega = (0,1)^3, \quad \Omega_0=(0.125,0.875)^3.
\end{align*}
The exact solution $\pmb{u}$ for \eqref{SPcurlLapcurlproblem} with given data is not known, so we can not compute $\pmb{u}-\pmb{u}_h$ directly. However, since the $\varepsilon$ we take is very small, according to \eqref{regularitySP2} and \eqref{regularitySP3}, we can estimate the error well by using the solution of the reduced problem. We point out that $\curl\pmb{u}$ has strong boundary layers. In fact, on the one hand, by the second boundary condition of \eqref{SPcurlLapcurlproblem} and simple computation we see that $\curl\pmb{u}|_{\partial\Omega}=0$ and $\curl\overline{\pmb{u}}|_{\partial\Omega}\neq0$. On the other hand, by \eqref{regularitySP2} we have $\fanshu[0]{\curl(\pmb{u}-\overline{\pmb{u}})}\rightarrow0$ as $\varepsilon\rightarrow0$, which means that $\curl\pmb{u}$ has strong boundary layers.

In Tables \ref{tableexample2weak} and \ref{tableexample2strong}, we compare the following relative errors
\begin{align*}
E_{L^2}(\overline{\pmb{u}},\pmb{u}_h),\quad E_{L^2}(\overline{\pmb{u}},\pmb{u}_h^0),\quad E_{\mathrm{curl}}(\overline{\pmb{u}},\pmb{u}_h),\quad E_{\mathrm{curl}}(\overline{\pmb{u}},\pmb{u}_h^0)
\end{align*}
and the corresponding convergence rates with weakly and strongly imposed boundary conditions. It can be seen from the tables that both relative errors converge at the optimal convergence rates for weakly imposed boundary conditions, while neither relative error converges optimally for strongly imposed boundary condition.

In Tables \ref{tableexample2weakOmega0} and \ref{tableexample2strongOmega0}, we compare the following relative errors
\begin{align*}
&E_{L^2,\Omega_0}(\overline{\pmb{u}},\pmb{u}_h),\quad E_{\mathrm{curl},\Omega_0}(\overline{\pmb{u}},\pmb{u}_h),\quad E_{\mathrm{g}\mathrm{c},\Omega_0}(\overline{\pmb{u}},\pmb{u}_h),\quad E_{\varepsilon,h,\Omega_0}(\overline{\pmb{u}},\pmb{u}_h)\\
&E_{L^2,\Omega_0}(\overline{\pmb{u}},\pmb{u}_h^0),\quad E_{\mathrm{curl},\Omega_0}(\overline{\pmb{u}},\pmb{u}_h^0),\quad E_{\mathrm{g}\mathrm{c},\Omega_0}(\overline{\pmb{u}},\pmb{u}_h^0),\quad E_{A,\Omega_0}(\overline{\pmb{u}},\pmb{u}_h^0),
\end{align*}
and the corresponding convergence rates with weakly and strongly imposed boundary conditions on a subdomain $\Omega_0$. Here the subscript $\Omega_0$ means that the related integral is calculated only on $\Omega_0$. Note that $\Omega_0$ is sufficiently far away from the layers. Therefore the exact solution is very smooth on $\Omega_0$. From the tables, we can observe that all relative errors converge at the optimal rates for the method with weakly imposed boundary condition. However, except $E_{\mathrm{g}\mathrm{c},\Omega_0}(\overline{\pmb{u}},\pmb{u}_h^0)$, no relative error converges at the optimal rate when the second boundary condition is imposed strongly. That is to say the presence of boundary layers may pollute the finite element method even far away from the boundary if the second boundary condition is imposed strongly. Similar phenomena are reported in the context of fourth order elliptic singular perturbation problem \cite{semper1992conforming,guzman2012spfourth}.

The above discussion illustrates the advantages of the weak treatment of the boundary conditions for the problem \eqref{SPcurlLapcurlproblem}.

\begin{table}[!htbp]\centering
\caption{Example 2: Rates of convergence for $\varepsilon=10^{-6}$ on $\Omega$ with weakly imposed boundary condition\label{tableexample2weak}}
\scalebox{1}[1]{
\begin{tabular}{*{5}{c@{\extracolsep{5pt}}@{\extracolsep{5pt}}}}\hline
$h$ &$E_{L^2}(\overline{\pmb{u}},\pmb{u}_h)$&  rate & $E_{\mathrm{curl}}(\overline{\pmb{u}},\pmb{u}_h)$&  rate \\
\hline
0.2165& 4.199e-02&      & 5.625e-02&      \\
0.1732& 2.745e-02& 1.90& 3.649e-02& 1.94\\
0.1443& 1.930e-02& 1.93& 2.553e-02& 1.96\\
0.1237& 1.429e-02& 1.95& 1.884e-02& 1.97\\
0.1083& 1.100e-02& 1.96& 1.447e-02& 1.98\\
0.0962& 8.720e-03& 1.97& 1.146e-02& 1.98\\
0.0866& 7.082e-03& 1.98& 9.294e-03& 1.99\\
\hline
\end{tabular}}
\end{table}

\begin{table}[!htbp]\centering
\caption{Example 2: Rates of convergence for $\varepsilon=10^{-6}$ on $\Omega$ with strongly imposed boundary condition\label{tableexample2strong}}
\scalebox{1}[1]{
\begin{tabular}{*{5}{c@{\extracolsep{5pt}}@{\extracolsep{5pt}}}}\hline
$h$ &$E_{L^2}(\overline{\pmb{u}},\pmb{u}_h^0)$&  rate & $E_{\mathrm{curl}}(\overline{\pmb{u}},\pmb{u}_h^0)$&  rate \\
\hline
0.2165& 9.971e-02&      & 2.780e-01&      \\
0.1732& 7.798e-02& 1.10& 2.476e-01& 0.52\\
0.1443& 6.404e-02& 1.08& 2.256e-01& 0.51\\
0.1237& 5.434e-02& 1.07& 2.086e-01& 0.51\\
0.1083& 4.720e-02& 1.06& 1.949e-01& 0.51\\
0.0962& 4.173e-02& 1.05& 1.836e-01& 0.51\\
0.0866& 3.739e-02& 1.04& 1.741e-01& 0.51\\
\hline
\end{tabular}}
\end{table}

\begin{table}[!htbp]\centering
\caption{Example 2: Rates of convergence for $\varepsilon=10^{-6}$ on $\Omega_0$ with weakly imposed boundary condition\label{tableexample2weakOmega0}}
\scalebox{1}[1]{
\begin{tabular}{*{9}{c@{\extracolsep{5pt}}@{\extracolsep{5pt}}}}\hline
$h$ &$E_{L^2,\Omega_0}(\overline{\pmb{u}},\pmb{u}_h)$&  rate & $E_{\mathrm{curl},\Omega_0}(\overline{\pmb{u}},\pmb{u}_h)$&  rate  & $E_{\mathrm{g}\mathrm{c},\Omega_0}(\overline{\pmb{u}},\pmb{u}_h)$ &  rate& $E_{\varepsilon,h,\Omega_0}(\overline{\pmb{u}},\pmb{u}_h)$ &  rate\\
\hline
0.2165& 2.175e-02&      & 3.685e-02&      & 3.420e-01&      & 3.669e-02&      \\ 0.1732& 1.405e-02& 1.96& 2.321e-02& 2.07& 2.631e-01& 1.18& 2.311e-02& 2.07\\
0.1443& 9.163e-03& 2.34& 1.566e-02& 2.16& 2.066e-01& 1.33& 1.560e-02& 2.16\\
0.1237& 6.650e-03& 2.08& 1.137e-02& 2.08& 1.741e-01& 1.11& 1.133e-02& 2.08\\
0.1083& 5.601e-03& 1.29& 9.413e-03& 1.42& 1.652e-01& 0.39& 9.373e-03& 1.42\\
0.0962& 4.417e-03& 2.02& 7.333e-03& 2.12& 1.443e-01& 1.15& 7.302e-03& 2.12\\
0.0866& 3.426e-03& 2.41& 5.819e-03& 2.19& 1.257e-01& 1.31& 5.794e-03& 2.20\\
\hline
\end{tabular}}
\end{table}

\begin{table}[!htbp]\centering
\caption{Example 2: Rates of convergence for $\varepsilon=10^{-6}$ on $\Omega_0$ with strongly imposed boundary condition\label{tableexample2strongOmega0}}
\scalebox{1}[1]{
\begin{tabular}{*{9}{c@{\extracolsep{5pt}}@{\extracolsep{5pt}}}}\hline
$h$ &$E_{L^2,\Omega_0}(\overline{\pmb{u}},\pmb{u}_h^0)$&  rate & $E_{\mathrm{curl},\Omega_0}(\overline{\pmb{u}},\pmb{u}_h^0)$&  rate  & $E_{\mathrm{g}\mathrm{c},\Omega_0}(\overline{\pmb{u}},\pmb{u}_h^0)$ &  rate& $E_{A,\Omega_0}(\overline{\pmb{u}},\pmb{u}_h^0)$ &  rate\\
\hline
0.2165& 5.741e-02&      & 5.021e-02&      & 3.779e-01&      & 5.031e-02&      \\
0.1732& 4.333e-02& 1.26& 3.348e-02& 1.82& 2.776e-01& 1.38& 3.362e-02& 1.81\\
0.1443& 3.474e-02& 1.21& 2.475e-02& 1.66& 2.098e-01& 1.53& 2.491e-02& 1.65\\
0.1237& 2.887e-02& 1.20& 1.956e-02& 1.53& 1.748e-01& 1.18& 1.971e-02& 1.52\\
0.1083& 2.735e-02& 0.40& 1.798e-02& 0.63& 1.660e-01& 0.39& 1.813e-02& 0.62\\
0.0962& 2.373e-02& 1.21& 1.515e-02& 1.45& 1.446e-01& 1.17& 1.529e-02& 1.45\\
0.0866& 2.093e-02& 1.19& 1.302e-02& 1.43& 1.259e-01& 1.31& 1.316e-02& 1.43\\
\hline
\end{tabular}}
\end{table}

\section*{Acknowledgments}
The authors would like to thank anonymous referees for their valuable comments. This work is supported in part by the National Natural Science Foundation of China grants NSFC 11871092, 12131005, and NSAF U1930402.
\bibliographystyle{siam}
\bibliography{SPQuadCurl}

\begin{thebibliography}{10}

\bibitem{boffi2013mixed}
{\sc D.~Boffi, F.~Brezzi, and M.~Fortin}, {\em Mixed finite element methods and
  applications}, vol.~44, Springer, 2013.

\bibitem{brenner2008fembook}
{\sc S.~C. Brenner and L.~R. Scott}, {\em The mathematical theory of finite
  element methods}, vol.~15 of Texts in Applied Mathematics, Springer, New
  York, third~ed., 2008.

\bibitem{cakoni2007variational}
{\sc F.~Cakoni and H.~Haddar}, {\em A variational approach for the solution of
  the electromagnetic interior transmission problem for anisotropic media},
  Inverse Probl. Imaging, 1 (2007), pp.~443--456.

\bibitem{cao2022error}
{\sc S.~Cao, L.~Chen, and X.~Huang}, {\em Error analysis of a decoupled finite
  element method for quad-curl problems}, J. Sci. Comput., 90 (2022), pp.~Paper
  No. 29, 25.

\bibitem{Chen2019HDGQuadCurl}
{\sc G.~Chen, J.~Cui, and L.~Xu}, {\em A hybridizable discontinuous {G}alerkin
  method for the quad-curl problem}, J. Sci. Comput., 87 (2021), pp.~Paper No.
  16, 23.

\bibitem{chen2008ifem}
{\sc L.~Chen}, {\em ifem: an innovative finite element methods package in
  matlab}, Preprint, University of Maryland,  (2008).

\bibitem{ern2016molification}
{\sc A.~Ern and J.-L. Guermond}, {\em Mollification in strongly {L}ipschitz
  domains with application to continuous and discrete de {R}ham complexes},
  Comput. Methods Appl. Math., 16 (2016), pp.~51--75.

\bibitem{ern2021finteI}
\leavevmode\vrule height 2pt depth -1.6pt width 23pt, {\em Finite elements
  {I}---{A}pproximation and interpolation}, vol.~72 of Texts in Applied
  Mathematics, Springer, Cham, [2021] \copyright 2021.

\bibitem{Vivette1986finite}
{\sc V.~Girault and P.-A. Raviart}, {\em Finite element methods for
  {N}avier-{S}tokes equations}, vol.~5 of Springer Series in Computational
  Mathematics, Springer-Verlag, Berlin, 1986.
\newblock Theory and algorithms.

\bibitem{guzman2012spfourth}
{\sc J.~Guzm\'{a}n, D.~Leykekhman, and M.~Neilan}, {\em A family of
  non-conforming elements and the analysis of {N}itsche's method for a
  singularly perturbed fourth order problem}, Calcolo, 49 (2012), pp.~95--125.

\bibitem{guzman2012afamily}
{\sc J.~Guzm\'{a}n and M.~Neilan}, {\em A family of nonconforming elements for
  the {B}rinkman problem}, IMA J. Numer. Anal., 32 (2012), pp.~1484--1508.

\bibitem{Hong2012DGFourthOrderCurl}
{\sc Q.~Hong, J.~Hu, S.~Shu, and J.~Xu}, {\em A discontinuous {G}alerkin method
  for the fourth-order curl problem}, J. Comput. Math., 30 (2012),
  pp.~565--578.

\bibitem{hu2020Simple}
{\sc K.~Hu, Q.~Zhang, and Z.~Zhang}, {\em Simple curl-curl-conforming finite
  elements in two dimensions}, SIAM J. Sci. Comput., 42 (2020),
  pp.~A3859--A3877.

\bibitem{hu2022afamily}
\leavevmode\vrule height 2pt depth -1.6pt width 23pt, {\em A family of finite
  element {S}tokes complexes in three dimensions}, SIAM J. Numer. Anal., 60
  (2022), pp.~222--243.

\bibitem{huang2022nonconforing}
{\sc X.~Huang}, {\em Nonconforming finite element stokes complexes in three
  dimensions}, arXiv preprint arXiv: arXiv:2007.14068v2,  (2022).

\bibitem{monk2012finite}
{\sc P.~Monk and J.~Sun}, {\em Finite element methods for {M}axwell's
  transmission eigenvalues}, SIAM J. Sci. Comput., 34 (2012), pp.~B247--B264.

\bibitem{Nilssen2001arobust}
{\sc T.~K. Nilssen, X.-C. Tai, and R.~Winther}, {\em A robust nonconforming
  {$H^2$}-element}, Math. Comp., 70 (2001), pp.~489--505.

\bibitem{semper1992conforming}
{\sc B.~Semper}, {\em Conforming finite element approximations for a
  fourth-order singular perturbation problem}, SIAM J. Numer. Anal., 29 (1992),
  pp.~1043--1058.

\bibitem{sun2011iterative}
{\sc J.~Sun}, {\em Iterative methods for transmission eigenvalues}, SIAM J.
  Numer. Anal., 49 (2011), pp.~1860--1874.

\bibitem{Sun2016MixedFemQuad}
\leavevmode\vrule height 2pt depth -1.6pt width 23pt, {\em A mixed {FEM} for
  the quad-curl eigenvalue problem}, Numer. Math., 132 (2016), pp.~185--200.

\bibitem{Sun2020C0interior}
{\sc Z.~Sun, F.~Gao, C.~Wang, and Y.~Zhang}, {\em A quadratic {$C^0$} interior
  penalty method for the quad-curl problem}, Math. Model. Anal., 25 (2020),
  pp.~208--225.

\bibitem{Wang2006modified}
{\sc M.~Wang, J.-c. Xu, and Y.-c. Hu}, {\em Modified {M}orley element method
  for a fourth order elliptic singular perturbation problem}, J. Comput. Math.,
  24 (2006), pp.~113--120.

\bibitem{zhang2022anew}
{\sc B.~Zhang and Z.~Zhang}, {\em A new family of nonconforming elements with
  $\pmb{H}(\mathrm{curl})$-continuity for the three-dimensional quad-curl
  problem}, preprint,  (2022).

\bibitem{zhang2019Hcurl2}
{\sc Q.~Zhang, L.~Wang, and Z.~Zhang}, {\em {$H({\rm curl}^2)$}-conforming
  finite elements in 2 dimensions and applications to the quad-curl problem},
  SIAM J. Sci. Comput., 41 (2019), pp.~A1527--A1547.

\bibitem{zhang2020curlcurl}
{\sc Q.~Zhang and Z.~Zhang}, {\em Curl-curl conforming elements on tetrahedra},
  arXiv preprint arXiv:2007.10421v2,  (2020).

\bibitem{Zhang2018MixedQuadCurl}
{\sc S.~Zhang}, {\em Mixed schemes for quad-curl equations}, ESAIM Math. Model.
  Numer. Anal., 52 (2018), pp.~147--161.

\bibitem{zheng2011nonconforming}
{\sc B.~Zheng, Q.~Hu, and J.~Xu}, {\em A nonconforming finite element method
  for fourth order curl equations in {$\Bbb{R}^{3}$}}, Math. Comp., 80 (2011),
  pp.~1871--1886.

\end{thebibliography}
\end{document}